\documentclass[msom,nonblindrev]{informs4} 

\OneAndAHalfSpacedXII 

\usepackage{natbib}
\usepackage{multirow}
\usepackage{booktabs}
 \bibpunct[, ]{(}{)}{,}{a}{}{,}%
 \def\bibfont{\small}%
 \def\bibsep{\smallskipamount}%
\TheoremsNumberedThrough     

\EquationsNumberedThrough    
\ECRepeatTheorems
\usepackage{amsmath}
\usepackage{graphicx}
\usepackage{bbm}
\usepackage{amssymb} 
\usepackage{amsfonts}
\usepackage{datetime}
\usepackage{mathtools}
\usepackage{comment}
\usepackage{xcolor}
\usepackage{threeparttable}
\usepackage{algorithm}
\usepackage{algorithmic}
\makeatletter
\renewcommand{\fnum@algorithm}{\small\textbf{\fname@algorithm~\thealgorithm}}
\makeatother
\usepackage{endnotes}
\usepackage{tikz}
\usepackage{caption}
\usepackage{subcaption} 
 
\usetikzlibrary{positioning, calc}

\tikzset{
    node_point/.style={circle, fill=black, inner sep=1.5pt, outer sep=2pt},
    hyedge/.style={draw=#1!80, fill=#1!20, thick, smooth cycle, tension=0.8, fill opacity=0.6}
}
\usepackage[colorinlistoftodos]{todonotes}
\definecolor{links}{RGB}{204,36,29}
\usepackage[colorlinks=true,breaklinks=true,bookmarks=true,urlcolor=links,citecolor=links,linkcolor=links,bookmarksopen=false,draft=false]{hyperref} 
\usepackage{cleveref}  

\newcommand{\M}{\mathcal{M}}

\def \x {\boldsymbol{x}}

\def \C {\mathcal{C}}

\def \w {\boldsymbol{w}}
\def \v {\boldsymbol{v}}

\def \bbmr {\mathbbm{R}}

\def \y {\boldsymbol{y}}

\makeatletter
\newcommand*{\rom}[1]{\expandafter\@slowromancap\romannumeral #1@}
\makeatother

\def \z {\boldsymbol{z}}
\def \w {\boldsymbol{w}}

\def\argmax{\mathop{\mathrm{argmax}}}
\def\argmin{\mathop{\mathrm{argmin}}}
\def \R {\mathcal{R}}

\DeclareMathOperator{\proj}{proj}

\DeclareMathOperator{\conv}{conv}

\DeclareMathOperator{\pr}{Pr}

\DeclareMathOperator{\for}{ for }
\DeclareMathOperator{\st}{s.t.}

\DeclareMathOperator{\G}{\mathrm{G}}
\DeclareMathOperator{\E}{\mathrm{E}}

 \renewcommand{\theARTICLETOP}{}

\begin{document}



\RUNTITLE{Assortment optimization under logit multi-purchase models}

\TITLE{Assortment Optimization under Logit-Based Multi-Purchase Choice Models}

\ARTICLEAUTHORS{
\AUTHOR{Taotao He, Zhongqi Wu, Yating Zhang}
\AFF{Antai College of Economics and Management, Shanghai Jiao Tong University, Shanghai, China, 200032}

Contact: hetaotao@sjtu.edu.cn (\href{https://taotaoohe.github.io/}{TH}), wuzhongqi@sjtu.edu.cn, zhang00yt@gmail.com (\href{https://zhangyt00.github.io/}{YZ})
}

\ABSTRACT{ \textbf{Problem definition:} 
We study assortment optimization under logit-based multi-purchase choice models. Because customers frequently purchase multiple distinct products together in a single transaction, incorporating such behavior into assortment decisions can unlock additional revenue opportunities. However, doing so is computationally challenging because the offered assortment induces a combinatorial set of feasible bundles that compete nonlinearly for customer demand. 
\textbf{Methodology/results:} 
We develop an exact assortment optimization framework for general multi-purchase patterns. We first introduce a hypergraph representation of Logit-MP models that unifies several models in the literature, including BundleMVL-K and Multivariate-MNL, and can be estimated from transaction data using a sparsity-based procedure. Building on this representation, we develop strong mixed-integer programming formulations that leverage recent advances from multilinear optimization and a perspective reformulation. The resulting formulations provide tighter linear programming relaxations than the prevalent Big-M approach. We further characterize classes of multi-purchase patterns, represented by different hypergraph structures, under which the formulations reduce to linear programs. Computational experiments demonstrate substantial improvements in both solution quality and scalability.
\textbf{Managerial implications:}
Using real transaction data, we show that a sparse hypergraph captures the dominant multi-purchase patterns while generating higher revenues. We also show that our formulations extend to heterogeneous customer populations and remain computationally viable, making the framework practical for large-scale retail assortment planning.

}


\KEYWORDS{Logit choice models; Multi-purchase behavior; Assortment optimization; Multilinear optimization; Perspective formulation}

\maketitle

%

\section{Introduction} 
Customers frequently purchase multiple products within a single transaction. This multi-purchase behavior has been consistently documented across a wide range of retail settings~\cite[e.g.,][]{Manchanda1999ShoppingBasket,russell2000analysis}. Accounting for such behavior is increasingly important in assortment planning because it enables firms to capture cross-selling opportunities, better predict customer demand, and increase total revenue~\citep{tulabandhula2023multi,jasin2024assortment}.


One prevalent approach to model multi-purchase behavior is logit-based models, referred to as Logit-MP models and formally introduced in Section~\ref{sec:logit:model} (see a comprehensive review of other multi-purchase models in Section~\ref{sec:literature}). These models generalize the Multinomial Logit (MNL) model by allowing alternatives to be subsets (or bundles) of products. Notable examples include the Multivariate MNL (MVMNL)~\citep{aurier2014multivariate,jasin2024assortment}, cardinality-restricted bundle logits (e.g., BundleMVL-K)~\citep{tulabandhula2023multi}. 
Theoretically, for bundles of cardinality two, these models admit an axiomatic characterization \citep{chambers2024correlated}. Empirically, these models have been evaluated and calibrated using the maximum likelihood estimation (MLE) method~\citep{aurier2014multivariate}. For large-scale datasets, recent works achieve estimation scalability through regularization \citep{vasilyev2025assortment} or sequential choice models with latent embeddings \citep{ruiz2020shopper}. Consequently, these models have been widely used in customer behavior analysis, including cross-category effects~\citep{song2006measuring,kwak2015analysis}, promotion effects~\citep{pan2025multicategory}, and price competition~\citep{richards2018retail}.

Building on their empirical success, Logit-MP models have recently been incorporated into assortment optimization~\citep[e.g.,][]{chen2022assortment, tulabandhula2023multi, jasin2024assortment}. However, existing exact optimization methods remain limited. 
Most tractable methods rely on restrictive assumptions on the feasible bundle set and develop approximation algorithms tailored to those structures. Specifically, \cite{tulabandhula2023multi} develop a heuristic for the BundleMVL-K model when customers purchase at most two products, while \cite{jasin2024assortment} derive a fully polynomial-time approximation scheme (FPTAS) for the MVMNL model under category restrictions that permit at most one product from each category. Such restrictions sometimes do not align with empirical purchase patterns: as illustrated by the bakery data in Figure~\ref{fig:general:motivation}, all products are in the same category and frequently observed bundles overlap around a small number of high-volume products, which cannot be represented by only cardinality- or category-based rules. To fully capture the revenue opportunities arising from these general bundle structures, it is essential to solve the assortment optimization problem under a general Logit-MP model to global optimality. 


\begin{figure}[htbp]
   \FIGURE
{\includegraphics[width=0.88\linewidth]{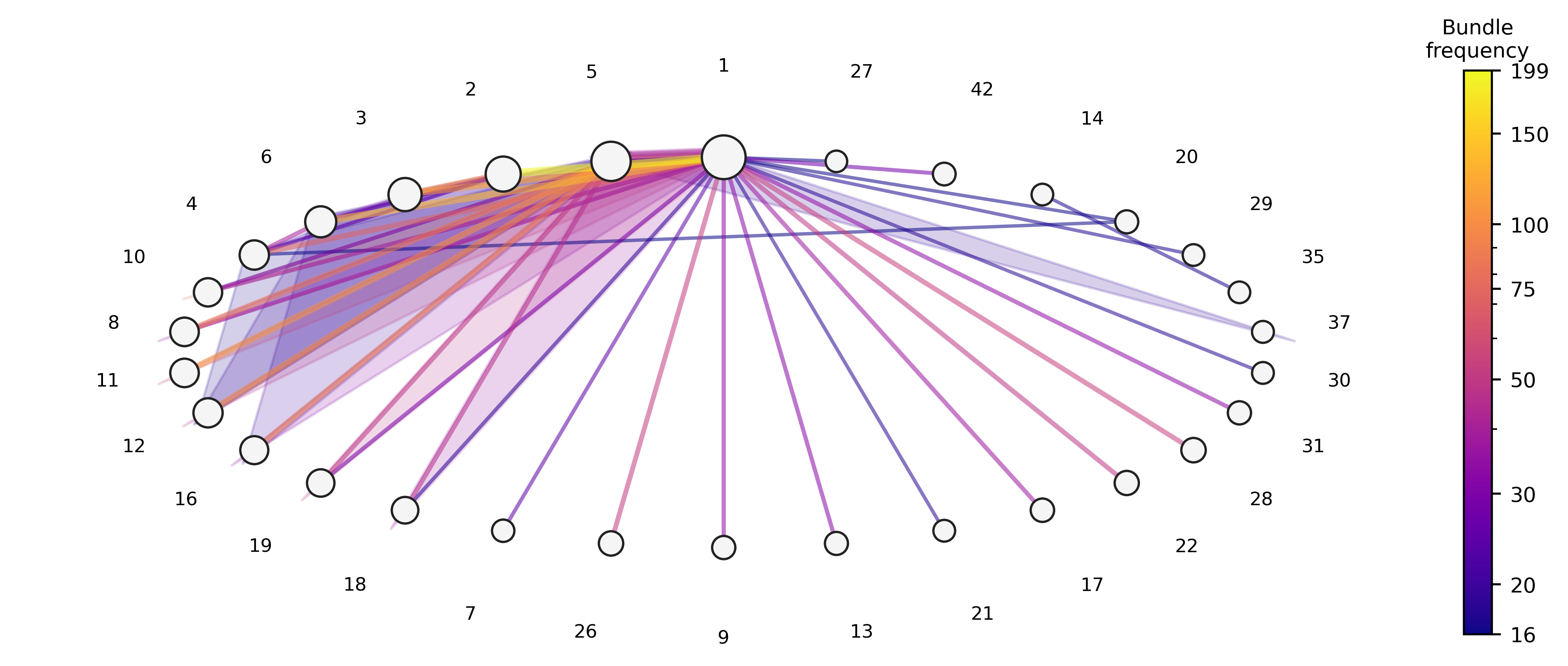}}
   {Top 50 most frequent multi-purchase bundles in the bakery data\label{fig:general:motivation}}
  {Each point represents a product, with point size proportional to its total sales volume. Each line or enclosed region represents a set of products purchased together in the same transaction, and the color indicates the observed bundle frequency. Product locations are chosen for visualization.}
\end{figure}

Developing such an exact solution methodology requires an explicit characterization of the relationship between assortment decisions and customer choice probabilities. This mapping is challenging for two reasons. First, the offered assortment induces a combinatorial collection of feasible bundles because each bundle is available only when all of its constituent products are offered. Second, the logit-based model couples all feasible bundles through a common normalization denominator, resulting in highly nonlinear choice probability functions. These two sources of nonlinearity make exact assortment optimization computationally challenging.

We address this challenge by developing a mixed-integer programming (MIP) formulation approach.  We first introduce a hypergraph representation of Logit-MP models, in which nodes represent products and hyperedges represent bundles considered by customers. This representation is general enough to encompass diverse purchase patterns, including those found in MVMNL and BundleMVL-K. Exploiting hidden structures in the hypergraph, we then derive strong linear relaxations for the combinatorial set defined by all feasible bundles.  Finally, we use a perspective reformulation to treat the logit-based choice probabilities. A key theoretical property of the perspective reformulation is that it preserves the strength of the underlying linear relaxation. Consequently, strong linear descriptions of feasible bundles immediately yield equally strong formulations of the choice probabilities.

We highlight the contributions of our paper as follows. 
\begin{enumerate}
\item  We introduce a hypergraph representation of Logit-MP models that accommodates arbitrary multi-purchase patterns and unifies several existing models. We further develop a sparsity-based estimation procedure to construct the hypergraph from transaction data. Our empirical study shows that sparse hypergraphs effectively capture the dominant multi-purchase patterns, while a semi-synthetic Sushi case study demonstrates additional revenue gains over models including MNL, the mixture of MNLs, and BundleMVL-2.

\item  We establish computational hardness of assortment optimization under general Logit-MP models, showing that even the unconstrained problem is NP-hard to approximate within a factor of $O(|V|^{1-\epsilon})$ for any fixed $\epsilon\in(0,1)$. This result  motivates the development of exact optimization methods for general multi-purchase patterns.


    \item 
Our proposed perspective formulation enables exact assortment optimization under general purchase patterns.  Beyond its generality, a key theoretical contribution is the strength of its linear programming (LP) relaxation. Specifically, we first show that its LP relaxation is tighter than that of the prevalent Big-M formulation. We then characterize classes of hypergraphs for which the LP relaxation is exact, including series-parallel graphs with pairwise bundles and kite-free $\beta$-acyclic hypergraphs with higher-order bundles. Finally, the formulation extends naturally to heterogeneous customer populations and robust optimization settings; see Section~\ref{sec:extend}.

   \item We develop a cutting-plane implementation for efficiently solving the
    proposed formulation. Computational experiments show that the
    perspective formulation outperforms standard Big-M and conic
    alternatives and demonstrate the computational value of the proposed
    structural inequalities. In particular, our formulations find a globally optimal assortment for $1000$ products within 800 seconds when the bundle size is no more than two. Furthermore, for more general patterns, incorporating structural inequalities into our models yields substantial speedups, significantly increasing the number of instances solved within a time limit. Finally, our formulation exhibits consistent performance in heterogeneous extensions and their robust counterparts.   
\end{enumerate}

\subsection{Literature Review}\label{sec:literature}


\textbf{Multi-Purchase Choice Models.}
Existing studies on multi-purchase behavior can be broadly categorized into three streams. The first stream models purchase decisions directly at the bundle level by extending discrete choice models to accommodate multiple products selected simultaneously. Early studies introduced multivariate logit and probit models to capture cross-product dependencies and complementarity effects~\citep{Manchanda1999ShoppingBasket,aurier2014multivariate}. Although Multivariate Probit models provide flexibility in modeling correlations, their estimation often requires computationally intensive simulation methods. In contrast, Multivariate MNL (MVMNL) models offer greater estimation tractability and have been widely applied in empirical studies of purchasing behavior and operational decisions~\citep[e.g.,][]{kwak2015analysis,richards2018retail,jasin2024assortment,pan2025multicategory}. Recent work further extends MVMNL to direct bundle-level choice models, such as the cardinality-restricted BundleMVL-$K$ model~\citep{tulabandhula2023multi}. We refer to these logit-based multi-purchase models as Logit-MP models.

A second stream models multi-purchase behavior using the structural utility maximization rule over bundles, often with budget or quantity constraints. Representative models include multi-discrete choice models~\citep{huh2022optimal}, multiple discrete-continuous extreme value models (MDC)~\citep{bhat2005multiple}, and threshold utility models~\citep{gallego2019threshold}. The Marginal Distribution Model (MDM)~\citep{sun2020unified} is a notable example that identifies potential product combinations through latent marginal distributions and has been applied to bundle design~\citep{li2022convex,sun2025partition}.

A third stream departs from bundle-level utility by focusing on ranking-based mechanisms. Under this paradigm, customers are assumed to rank individual items and select the top $K$ products with the highest latent utilities. For instance, \cite{bai2024assortment} assume product utilities follow an MNL distribution and derive a logit-like probability form for the top $K$ items. It is important to distinguish this from the first stream: while the resulting formula also appears logit-like, its derivation is rooted in the ordering of individual items rather than the utility assigned to a fixed bundle. Other extensions generalize this by either applying order statistics to select the top $K$ products under general utility functions~\citep{abdallah2024multi}, or by assuming that customers pick the top $K$ available items from a fixed preference list~\citep{lin2026intra}.

Among these approaches, Logit-MP models provide an attractive balance between behavioral flexibility and estimation tractability. However, their integration into assortment optimization remains challenging because general multi-purchase patterns create complex interactions between assortment decisions and bundle-level choice probabilities.

\textbf{Assortment Optimization under Multi-Purchase Choice Models.} 
An increasing body of literature has integrated multi-purchase choice models into assortment planning. However, multi-purchase behavior introduces significant optimization challenges. Under the Logit-MP model, most studies focus on heuristic or approximation algorithms. For instance, \cite{jasin2024assortment} proposed an FPTAS under a special case of feasible bundles defined by the MVMNL model. Similarly, \cite{tulabandhula2023multi} rely on heuristics to obtain a good assortment decision under the BundleMVL-K model. Our work contributes to this stream of literature by deriving exact MIP formulations for assortment optimization under Logit-MP models with general purchase patterns.  

For the second and third streams of multi-purchase literature, the corresponding optimization problems are more difficult. Under utility maximization with budget constraints, the intractability stems from the complex bi-level structure, i.e., customers' choices and operational decisions. Under ranking-based models, the difficulty lies in formulating the relation between operational decisions and top-K rules. As a result, the literature mainly relies on approximation algorithms such as MDM~\citep{mao2025optimizing}, top-K MNL utilities~\citep{bai2024assortment}, and top-K preference lists~\citep{yang2025assortment}.

Our methodology contributes to a growing stream of research that leverages integer programming techniques to solve optimization problems involving customer choice models. Recent work has focused on linearizing the fractional structure of logit-based choices, including convexification results for the MNL model in a quick-commerce setting~\citep{chen2025integer}, and strong formulations for the mixture of MNL models by using second-order cones~\citep{sen2018conic} and submodularity~\citep{atamturk2020submodularity,kilincc2025conic}.~\cite{akchen2025exact} propose a perspective formulation for a logit-based share-of-choice product design problem, with subsequent extensions to history-dependent effects \citep{he2024assortment} and resource allocation for crime prevention  \citep{he2025proactive}. Under the random utility model,~\cite{bertsimas2019exact} propose an MIP formulation that is integral for single-segment problems, while \cite{ma2023assortment} and \cite{khalid2025assortment} provide tighter inequalities for multiple segments. Finally, modeling tools for piecewise linear functions and disjunctive constraints have also been used to address pricing problems under a representative customer model~\citep{yan2022representative}, assortment optimization under the mixture-of-nested-logit model~\citep{wang2025assortment}, and decision tree-based assortment optimization \citep{akchen2021assortment}. Our work contributes to assortment optimization literature by providing an exact formulation for the node-to-bundle activation induced by arbitrary multi-purchase patterns.

\subsection{Structure}\label{sec:structure}
The rest of the paper is organized as follows. In Section~\ref{sec:hypergraph}, we present Logit-MP models and their hypergraph representation. Section~\ref{sec:model} introduces our MIP formulations and their theoretical properties. Section~\ref{sec:estimation} constructs sparse hypergraphs from transaction data and evaluates their predictive and prescriptive performance. We report numerical results on synthetic data in Section~\ref{sec:app} and model extensions in Section~\ref{sec:extend}. We conclude the paper in Section~\ref{sec:conclude}. All proofs are given in the e-companion.



\section{Problem Formulation}\label{sec:hypergraph} 
This section first introduces the logit-based multi-purchase model and its hypergraph representation in Sections~\ref{sec:logit:model} and~\ref{sec:hypergraph:rep}, respectively. Section~\ref{sec:assortment:opt} then formulates the corresponding assortment optimization problem and establishes its computational complexity.

\subsection{Logit-Based Multi-Purchase Choice Model}\label{sec:logit:model}

Suppose the product set is $V=\{1,\dots,N\}$, where $N\in \mathbbm{Z}_+$. Logit-based multi-purchase models extend the alternative space of the MNL model from individual products to product bundles. Specifically, given an assortment $S\subseteq V$, a customer may purchase a bundle $e\subseteq S$ rather than a single product. In principle, customers may consider any bundle in the power set $2^V$. However, empirical evidence suggests that customers typically consider only a limited subset of possible bundles due to cognitive constraints and search  costs~\citep{hauser1990evaluation,roberts1991development,hauser2014consideration}. We therefore define $\E\subseteq 2^V$ as the \textit{consideration set} of bundles that customers may consider when all products are available. We require $\E$ to contain all singleton bundles, ensuring that the standard single-purchase setting remains feasible. Beyond this requirement, $\E$ is unrestricted and can capture general multi-purchase patterns.

Each feasible bundle $e \in \E$ is associated with a random utility. Existing Logit-MP models typically adopt one of two utility specifications. The first is the full utility model, which assigns an independent utility parameter to each feasible bundle~\citep{benson2018discrete}:
\begin{equation*}
U_e = u_e + \epsilon_e, \tag{\textsc{Full-Utility}}\label{eq:u:bundle}
\end{equation*}
where $u_e$ is the deterministic utility and $\epsilon_e$'s are i.i.d. from a Gumbel distribution. This specification is flexible because it allows arbitrary bundle-level utilities. However, the number of utility parameters grows exponentially with the number of products if all bundles are included. An alternative approach is to impose a decomposable structure on the deterministic utility $u_e$, i.e., 
\begin{equation*}
U_e = \sum_{i\in e} \alpha_i + \sum_{\{i,j\}\subseteq e} \beta_{ij} + \epsilon_e, \tag{\textsc{Pairwise-Utility}}\label{eq:u:pairwise}
\end{equation*}
where $\alpha_i$ is the base utility of each product $i$ and $\beta_{ij}$'s are pairwise interaction effects among products~\citep{russell2000analysis,vasilyev2025assortment}. The additive form makes \eqref{eq:u:pairwise} more scalable when the number of possible bundles is large. Therefore, it has been widely used in empirical marketing studies~\citep[e.g.,][]{hruschka1999cross,aurier2014multivariate,richards2018retail}.



Given an offered assortment $S$, customers can only purchase bundles that are both considered and fully offered, {i.e.}, $\E(S)= \{e \in \E \mid e \subseteq S \}$. 
Under the Gumbel distribution assumption, the probability of purchasing bundle $e$ is
\begin{equation*}\label{eq:choice}
    P_{\E}(e,S) = \frac{\exp(u_e)}{1+ \sum_{c\in \E(S)}\exp(u_c)}  \quad \text{ if } e\in \E(S),  \tag{\textsc{Logit-MP}}  
\end{equation*}
where $u_e$ is the utility of bundle $e$, and
the outside option has deterministic utility normalized to zero, which gives the constant one in the denominator. 
This provides a unified framework for existing models, where different models correspond to different structural restrictions on $\E$. For example, the MVMNL model restricts $\E$ by allowing each bundle to contain at most one product from each category~\citep{jasin2024assortment}. The BundleMVL-$K$ model limits $\E$ to bundles with cardinality at most $K$~\citep{tulabandhula2023multi}. 

\subsection{A Hypergraph Representation}\label{sec:hypergraph:rep}

The consideration set formulation has a natural hypergraph representation. We next introduce this representation and explain how it links feasible purchase bundles to assortment decisions. A hypergraph generalizes a graph by allowing an edge to connect an arbitrary subset of nodes. Such subsets are referred to as \textit{hyperedges}. We define a \textit{multi-purchase hypergraph} $\G(V,\E)$, where each node in $V$ represents a product and each hyperedge $e\in\E$ represents a bundle that may be purchased by customers. The weight of hyperedge $e$ is its attraction value, denoted by $v_e=\exp(u_e)$.  Singleton hyperedges correspond to single-product purchases, while hyperedges with cardinality greater than one capture multi-purchase alternatives.


This hypergraph representation allows us to visualize different Logit-MP models. For example, Figure~\ref{fig:mvmnl} represents the MVMNL model~\citep{jasin2024assortment}, where customers choose at most one product from the $k$-th category for $k=1,2, 3$.  
Figure~\ref{fig:bundle} illustrates the BundleMVL-2 model introduced in~\cite{tulabandhula2023multi}, where customers buy at most two products. Figure~\ref{fig:ric} shows a general multi-purchase pattern, including disjoint, nested, and chain-structured choices. That is, choices $e_1$ and $\{5\}$ are disconnected. $e_1,e_2$ and $e_3$ form a chain included in $e_0$.

\captionsetup[subfigure]{font=small}  
\begin{figure}[htbp]
    \centering
    \caption{Hypergraph representations of multi-purchase behavior under different consideration sets. The nodes $i$ represent products, and colored regions represent hyperedges (bundles). }
    \label{fig:hypergraph}
 
    \begin{subfigure}[b]{0.3\textwidth}
        \centering
    \begin{tikzpicture}[scale=0.9] 
    \coordinate (n1) at (0, 1.5);
    \coordinate (n2) at (1.5, 1.5);
    \coordinate (n3) at (3.0, 1.5);
    \coordinate (n4) at (0, 0.1);
    \coordinate (n5) at (1.5, 0.1);
    \coordinate (n6) at (3.0, 0.1);

    
    \draw[hyedge=red] (1.5,1.65) ellipse (2 and 0.4);
\node[font=\small, red] at (2.2,1) {$e_1$};

    \draw[hyedge=blue] plot [smooth cycle, tension=0.8] coordinates {
        (-0.3, 1.8)   
        (0.5, 1.8)    
        (0.8, 0.8)    
        (2.0, 0.3)    
        (3.5, 0.3)    
        (3.3, -0.3)   
        (1.2, -0.3)   
        (0.6, 0.3)    
        (-0.3, 1.2)   
    };
\node[font=\small, blue] at (0.2,1) {$e_2$};

    \foreach \i in {1,...,6} \node[node_point] at (n\i) {};

    \node[above, font=\small] at (n1) {$1$};
    \node[above, font=\small] at (n2) {$2$};
    \node[above, font=\small] at (n3) {$3$};
    
    \node[below, font=\small] at (n4) {$4$};
    \node[below, font=\small] at (n5) {$5$};
    \node[below, font=\small] at (n6) {$6$};

\end{tikzpicture}
        \caption{MVMNL}
        \label{fig:mvmnl}
    \end{subfigure}
    \hspace{0.5em}
    \begin{subfigure}[b]{0.3\textwidth}
        \centering
 \begin{tikzpicture}[scale=0.9]

\coordinate (n1) at (0, 1);
\coordinate (n2) at (1.5, 0.2);
\coordinate (n3) at (1, -0.8);
\coordinate (n4) at (-1, -0.8);
\coordinate (n5) at (-1.5, 0.2);

\draw[hyedge=teal]
  plot coordinates {
    (-0.5, 1.3) (0.3, 1.3) (1.9, 0) (1.2, 0)
  };
\node[font=\small, teal] at (0.7,0.7) {$e_1$};

\draw[hyedge=purple]
  plot coordinates {
    (2, 0.4) (1.5, -1) (0.8, -1) (1.2, 0.5)
  };
\node[font=\small, purple] at (1.25,-0.3) {$e_2$};

\draw[hyedge=orange]
  plot coordinates {
     (-1.1, -0.6) (1.1, -0.6) (1.1, -1.2) (-1.1,-1.2)
  };
\node[font=\small, orange] at (0,-0.8) {$e_3$};

\draw[hyedge=brown]
  plot coordinates {
     (-2, 0.3) (-1, 0.3) (-0.5,-1) (-1.5,-1)
  };
\node[font=\small, brown] at (-1.25,-0.3) {$e_4$};

\draw[hyedge=black]
  plot coordinates {
     (-0.5, 1.3) (0.5, 1.3) (-1.3, 0) (-2.2, 0)
  };
\node[font=\small, black] at (-0.7,0.6) {$e_5$};

\foreach \i in {1,...,5}
  \node[node_point] at (n\i) {};

\node[above, font=\small]       at (n1) {$1$};
\node[right, font=\small]       at (n2) {$2$};
\node[below right, font=\small] at (n3) {$3$};
\node[below left, font=\small]  at (n4) {$4$};
\node[left, font=\small]        at (n5) {$5$};

\end{tikzpicture}
\caption{BundleMVL-2}\label{fig:bundle}
\end{subfigure}
    \hspace{0.5em}
    \begin{subfigure}[b]{0.3\textwidth}
        \centering
        \begin{tikzpicture}[scale=0.9] 
\coordinate (n1) at (0.3,2);
\coordinate (n2) at (1.3,2);
\coordinate (n3) at (2.6,2);
\coordinate (n4) at (3.6,3.1);
\coordinate (n5) at (4.2,2);
\coordinate (n6) at (4.3,1.2); 

\draw[hyedge=red, dashed, fill opacity=0.12]
plot coordinates {
    (0,2.6) (3,2.7)(4.7,2.6)  
   (4.7,1) (3,0.9)  (0,1)
};

\node[font=\small, red] at (2,3) {$e_0$};

\draw[hyedge=cyan, fill opacity=0.25]
plot coordinates {
    (0,2.3) (1.4,2.3)   (1.4,1.6) (0,1.6)
};
\node[font=\small, cyan] at (0.5,1.24) {$e_1$};

\draw[hyedge=violet, fill opacity=0.25]
plot coordinates {
    (1.1,2.3) (2.8,2.3) (2.8,1.6) (1.1,1.6)
};
\node[font=\small, violet] at (2,1.2) {$e_2$};

\draw[hyedge=magenta, fill opacity=0.25]
plot coordinates {
     (2.4,2.3) (3.5, 3.3)(4.5,2.3) (4.5,1.6) (2.4,1.6)
};
\node[font=\small, magenta] at (3.5,1.2) {$e_3$};

\foreach \i in {1,...,6}
  \node[node_point] at (n\i) {};

\node[left,  font=\small] at (n1) {$1$};
\node[left, font=\small] at (n2) {$2$};
\node[left, font=\small] at (n3) {$3$};
\node[left, font=\small] at (n4) {$4$};
\node[right, font=\small] at (n5) {$5$};
\node[left, font=\small] at (n6) {$6$};

\end{tikzpicture}
 \caption{General patterns}\label{fig:ric}
    \end{subfigure}
\end{figure}
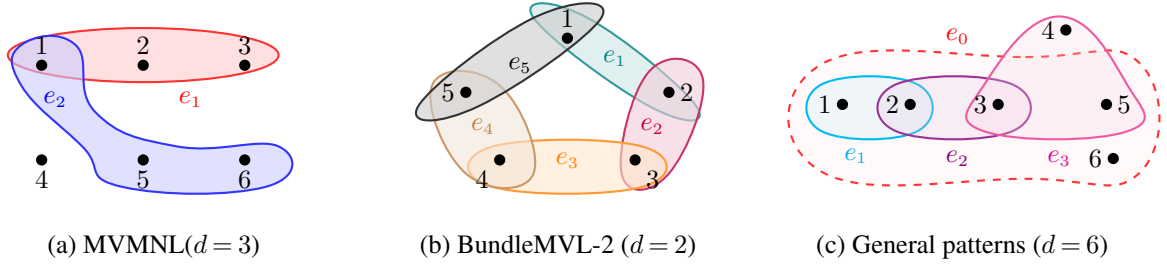

The hypergraph representation also provides a new interpretation of the assortment decision. An assortment $S\subseteq V$ selects a subset of product nodes and activates all hyperedges fully contained in $S$, namely, $\E(S) = 2^S \cap \E$. This node-to-hyperedge activation view highlights why assortment optimization under Logit-MP differs from the standard MNL problem: the value of an assortment is determined by the collection of bundles it induces, rather than by individual products alone. This interpretation will be used in the next subsection to formulate the assortment optimization problem.

\subsection{The Assortment Optimization Problem}\label{sec:assortment:opt}
We now formulate the Logit-MP assortment problem. Given a multi-purchase hypergraph $\G(V,\E)$, the retailer chooses an assortment $S\subseteq V$, or equivalently a binary vector $\x\in\{0,1\}^{|V|}$, where $x_i=1$ if and only if product $i$ is offered. Given an assortment decision $\x$, a bundle $e$ is available only when all of its constituent products are offered. Hence, the hyperedge activation indicator is $\prod_{i\in e}x_i$. To separate this combinatorial activation structure from the logit normalization term, we introduce the variable $y_e$ such that $v_e y_e$ represents the purchase probability of bundle $e$, namely

\begin{equation}\label{eq:ye}
    y_e = \frac{\prod_{i\in e}x_i}{1+\sum_{c\in\E}v_c\prod_{j\in c}x_j}, \for e\in \E.
\end{equation} 
Similarly, the no-purchase probability is 
\begin{equation}\label{eq:rho}
\rho = \frac{1}{1+\sum_{c\in\E}v_c\prod_{j\in c}x_j}.
\end{equation}
Let $r_e$ be the revenue when bundle $e$ is purchased. The retailer solves the following problem to maximize the total expected revenue.
\begin{equation*}
    \max_{\x\in \{0,1\}^{|V|} }\biggl\{ \sum_{e\in\E} r_e v_e y_e  \biggm| \x \in \mathcal{X},\  ~\eqref{eq:ye},\ \eqref{eq:rho} \biggr\} \tag{\textsc{Logit-OPT}}\label{eq:assortment-single}
\end{equation*} 
where $\mathcal X\subseteq [0,1]^{|V|}$ is a polyhedron representing the feasible set of assortment decisions, such as cardinality or precedence constraints.

When $\E$ contains only singleton bundles, the problem reduces to the standard MNL assortment optimization problem, for which optimal structures and LP formulations have been established~\citep{talluri2004revenue,sumida2021revenue,chen2025integer}. In contrast, allowing multi-product bundles makes the availability of each option depend on multiple products simultaneously, leading to a nonlinear integer optimization problem. Unlike the classical MNL assortment problem, where strong structural results are available, Logit-MP assortment optimization is fundamentally harder. We show that even the unconstrained problem is NP-hard to approximate within a factor of $O(|V|^{1-\epsilon})$.

\begin{theorem}\label{them:apx}
The assortment optimization problem~\eqref{eq:assortment-single} is NP-hard to approximate within a factor of $O(\vert V \vert^{1-\epsilon})$ for any fixed $\epsilon \in (0,1)$, even when there is no constraint.
\end{theorem}  

The hardness result is consistent with the computational challenges observed in existing studies. Recent literature has primarily focused on designing approximation or heuristic algorithms tailored to a specific, highly structured consideration set $\E$. For example, under the MVMNL model, \cite{chen2022assortment} design a constant-approximation algorithm limited to a two-category setting, while \cite{jasin2024assortment} propose an FPTAS for general categories. For the BundleMVL-2 model, \cite{tulabandhula2023multi} propose a heuristic algorithm. In contrast, our hypergraph representation allows the consideration set to be arbitrary. Theorem~\ref{them:apx} rules out strong polynomial-factor approximation guarantees for general bundle supports. We therefore develop exact MIP formulations that accommodate arbitrary multi-purchase hypergraphs as well as standard assortment constraints in the next section.

\section{MIP Formulations}\label{sec:model} 


In this section, we develop mixed-integer programming (MIP) formulations for the Logit-MP assortment problem. The key challenge is that~\eqref{eq:assortment-single} involves two distinct sources of nonlinearity. First, an assortment decision determines which bundles are available; that is, a bundle $e$ can be purchased if and only if all products in $e$ are offered. Thus, the availability of each bundle is represented by a multilinear monomial in the assortment variables. Second, conditional on the available bundles, the Logit-MP model assigns choice probabilities through a fractional form.

We address these two nonlinearities separately. We first formulate the bundle-activation problem through a generic linear relaxation for the monomial set
\begin{equation}\label{eq:mg}
\M_{\G}
:=
\biggl\{
\z \in [0,1]^{|\E|}
\biggm|
z_e=\prod_{i\in e}x_i,\;
e\in\E,\;
\x\in\{0,1\}^{|V|}
\biggr\},
\end{equation}
where $z_e$ indicates whether bundle $e$ is available under the offered assortment. For singleton hyperedges, that is, $e = \{i\}$ for some $i \in V$, we identify $z_e$ with the product-availability variable $x_i$. We, then, use the no-purchase probability as a perspective variable to transform this bundle-activation formulation into a formulation describing choice probabilities $\y$. 
Section~\ref{sec:strong:model} develops the perspective formulation and specializes the underlying bundle-activation relaxation of $\M_{\G}$ to different hypergraph structures; Section~\ref{sec:naive:model} compares it with standard Big-M and conic alternatives.



\subsection{Perspective Formulations}\label{sec:strong:model} 
Our approach provides flexibility in modeling the bundle-activation set $\M_{\G}$. We assume access to a linear relaxation for $\M_{\G}$ described by 
\begin{equation}\label{eq:oracle}
\R :=\bigl\{ \z \in [0,1]^{|\E|}  \bigm| \exists \w \text{ s.t. }   \BFA\z + \BFB \w \leq \BFb\bigr\}, 
\end{equation}
where $\w$ is an auxiliary variable, and $\BFA,\BFB$ and $\BFb$ have compatible dimensions. 
The bundle-activation relaxation $\R$ is said to be \textit{exact} if, when combined with the binary constraints $z_i \in \{0,1\}$ for all $i \in V$, it yields an MIP formulation for $\M_{\G}$. It is said to be \textit{integral} if for every vertex $\z$ of $\R$, we have $z_e \in \{0,1\}$ for all $e \in \E$.

Given a bundle-activation relaxation $\mathcal{R}$, we construct a formulation as follows
\begin{subequations}\label{eq:pers:oracle} 
    \begin{alignat}{3}
    \max \limits_{\x\in \mathcal{X} } & \sum_{e\in \E}r_e v_e y_e  \tag{\textsc{Pers}}  \label{model:pers} \\
\st  &  \y \in \rho\cdot \R  \text{ and } \rho \geq 0 \label{eq:pers:oracle-1}  \\
    & \, \rho + \sum_{e \in \E}v_ey_e = 1 \label{eq:pers:oracle-2} \\
    & x_i \in \{0,1\} && \quad  \for i \in V \label{eq:pers:oracle-3} \\
    & x_i \geq  y_i + \sum_{e \in \E | i\in e } v_e y_e  + \sum_{e \in \E | i\notin e} v_e \max\{0, y_i +y_e -\rho\}   && \quad \for i \in V \label{eq:pers:oracle-4}\\
    & x_i \le  y_i + \sum_{e \in \E | i\in e } v_e y_e + \sum_{e \in \E | i\notin e} v_e \min\{y_i, y_e \}   && \quad \for i \in V. \label{eq:pers:oracle-5} 
\end{alignat}
\end{subequations} 
For notational simplicity, we here write $y_i=y_{\{i\}}$ for singleton bundles.
Constraint~\eqref{eq:pers:oracle-1} scales the relaxation via the nonnegative no-purchase variable $\rho$, namely
\[
\rho \cdot \R := \bigl\{ \y \in [0,\rho]^{|\E|}\bigm| \exists \w \text{ s.t. } \BFA \y + \BFB \w \leq \BFb \rho   \bigr\}.
\] 
This scaling is valid because all purchase probabilities share the same denominator as the no-purchase probability. 
Constraint~\eqref{eq:pers:oracle-2} enforces normalization of choice probabilities. The binary constraint~\eqref{eq:pers:oracle-3} is the assortment decision. Finally,  constraints~\eqref{eq:pers:oracle-4}–\eqref{eq:pers:oracle-5} link the assortment decision $\x$ with the choice probabilities $(\rho,\y)$ to ensure consistency with~\eqref{eq:assortment-single}. Here, constraints~\eqref{eq:pers:oracle-4}–\eqref{eq:pers:oracle-5} are written compactly using max/min operators; equivalently, they represent finite families of linear inequalities. To solve the formulation efficiently, we implement these constraints through a cutting-plane procedure, described in~\ref{ec:sec:algo}.


The key theoretical feature of the construction is that it preserves the strength of the linear relaxation $\R$. In particular, when the tightest relaxation is provided (i.e., $\R$ is exact and integral), the LP relaxation~\eqref{model:pers}, obtained by dropping the binary requirement on $\x$, exactly solves the assortment optimization problem~\eqref{eq:assortment-single}. The next result formalizes this preservation property. Its proof relies on the perspective transformation developed in Theorem 1 of~\cite{he2024convexification}. Hence, we refer to~\eqref{model:pers} as the \textit{perspective formulation}.



\begin{theorem}\label{theorem:main}
For any hypergraph $\G$ and assortment feasible set $\mathcal{X}$, the formulation~\eqref{model:pers} exactly solves~\eqref{eq:assortment-single} if the bundle-activation relaxation $\R$ is exact. If the assortment problem is unconstrained and $\R$ is exact and integral, the LP relaxation of~\eqref{model:pers} exactly solves~\eqref{eq:assortment-single}.
\end{theorem} 
Last, we remark that the perspective formulation~\eqref{model:pers} is closely related to Remark 3  of \cite{he2024convexification}, which establishes an equivalence between the convex hull of the monomial set $\M_{\G}$ and that of the choice probability set $\{(\rho,\y)\mid \eqref{eq:ye}, ~\eqref{eq:rho},\for e\in \E, \x\in \{0,1\}^V\}$. Rather than pursuing a complete polyhedral description, we focus on developing MIP formulations that are compact, tight, and directly usable for solving assortment optimization problems. 


\subsubsection{Bundle-Activation Relaxation}\label{sec:activation}
Once the perspective formulation is established, the remaining task is to construct strong relaxations of $\M_{\G}$. 
This section builds three constructions tailored to different purchase structures. Recursive McCormick inequalities provide a generic exact formulation for arbitrary hypergraphs. For pairwise purchase patterns, odd-cycle inequalities exploit cyclic overlaps among bundles; for higher-order patterns, running-intersection inequalities exploit structured overlaps among hyperedges. These inequalities yield integral activation relaxations for certain classes of hypergraphs and, consequently, exact LP formulations of the corresponding Logit-MP assortment problems.

\tikzset{
    base/.style = {draw, rounded corners, minimum width=1.2cm, minimum height=0.6cm, align=center, font=\large\bfseries, line width=1pt},
    arrow/.style = {->, >=stealth, line width=1.2pt},
    rootRed/.style = {base, draw=red, text=red!80!black},
    rootBlue/.style = {base, draw=blue, text=blue!80!black},
    rootGreen/.style = {base, draw=green!60!black, text=green!40!black},
    rootPurple/.style = {base, draw=violet, text=violet},
    leaf/.style = {base, draw=black, text=black}
}
\textbf{General Multi-Purchase Patterns}\label{sec:rmc} Recursive McCormick inequalities decompose the activation of a high-order bundle into a sequence of smaller and easier availability checks. Specifically, a monomial term is reduced to a sequence of bilinear terms and then each bilinear term is linearized by McCormick inequalities~\citep{mccormick1976computability}.
For a bundle with size two, e.g., $e=\{i,j\}$, the indicator relation $z_e = z_i z_j$ can be exactly represented by:
\[ 
\max \{0, z_i + z_j -1\} \le z_e \le \min\{ z_i, z_j\}.
\] 
For a bundle $e$ containing more than two products, we apply this idea recursively. More specifically, we split $e$ into two disjoint subsets $p \subset e$ and $q \subset e$, and then relax the relation $z_e = z_pz_q$. This procedure is repeated until all subsets reduce to singletons. In this way, a high-order indicator can be represented through a sequence of bilinear indicators, each linearized via McCormick inequalities. The next example illustrates this using the hypergraph in Figure~\ref{fig:mvmnl}. 

\begin{example}
In Figure~\ref{fig:mvmnl}, customers consider two bundles, $\{1,2,3\}$ and $\{1,5,6\}$. The recursive decomposition can be visualized using a binary tree with two root nodes, as shown in Figure~\ref{fig:rmc:mvmnl}. The bundle set $\{1,2,3\}$ is first decomposed into bundles $\{1\}$ and $\{2,3\}$, and then $\{2,3\}$ is further split into $\{2\}$ and $\{3\}$. The bundle $\{1,5,6\}$ is decomposed in a similar way. Based on the decomposition structure in Figure~\ref{fig:rmc:mvmnl}, the recursive McCormick formulation for the multi-purchase hypergraph in Figure~\ref{fig:mvmnl} is:
    \begin{align*}
       \max \{0, z_{23}+ z_1 -1\} &\le z_{123} \le \min\{ z_{23}, z_{1} \} \quad &\max \{0, z_{56}+ z_1 -1\} \le z_{156} \le \min\{ z_{56}, z_{1} \}\\
       \max \{0, z_{2}+ z_3 -1\} &\le z_{23} \le \min\{ z_{2}, z_{3} \} \quad 
       &\max \{0, z_{5}+ z_6 -1\} \le z_{56} \le \min\{ z_{5}, z_{6} \},
\end{align*} 
where $z_{23}$ and $z_{56}$ are auxiliary variables introduced by the recursive decomposition.     \hfill \Halmos
\end{example}
\begin{figure}[htbp]
    \centering  \caption{Binary-tree representation of the recursive McCormick decomposition for the multi-purchase hypergraph in Figure~\ref{fig:mvmnl}}\label{fig:rmc:mvmnl}
        \begin{tikzpicture}[scale=0.7, transform shape, node distance=1.2cm and 0.5cm]
            \node[rootRed] (R1) at (0,0) {\{1,2,3\}};
            \node[rootBlue] (R2) at (4,0) {\{1,5,6\}};

            \node[rootGreen] (sub23) at (0, -2) {\{2,3\}};
            \node[leaf] (leaf1) at (2, -2) {\{1\}}; 
            \node[rootPurple] (sub56) at (4, -2) {\{5,6\}};

            \node[leaf] (l2) at (-0.8, -4) {\{2\}};
            \node[leaf] (l3) at (0.8, -4) {\{3\}};
            
            \node[leaf] (l5) at (3.2, -4) {\{5\}};
            \node[leaf] (l6) at (4.8, -4) {\{6\}};

            \draw[arrow] (R1) -- (sub23);
            \draw[arrow] (R1) -- (leaf1);
            \draw[arrow] (R2) -- (leaf1);
            \draw[arrow] (R2) -- (sub56);
            
            \draw[arrow] (sub23) -- (l2);
            \draw[arrow] (sub23) -- (l3);
            \draw[arrow] (sub56) -- (l5);
            \draw[arrow] (sub56) -- (l6);
        \end{tikzpicture}
\end{figure}
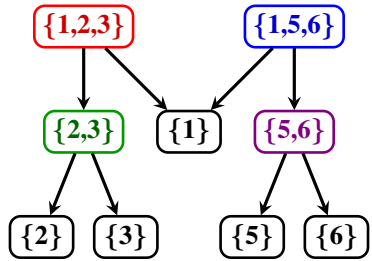

We now formally define the decomposition procedure. For each multi-purchase hypergraph $\G$, we can construct a binary tree, possibly with multiple roots, to decompose hyperedges $\E$, denoted as $T(\G)$. For each non-leaf node $t$ of $T(\G)$, its left (resp. right) child is denoted as $t_\ell$ (resp. $t_r$), where children $t_{\ell}$ and $t_r$ form a partition of their parent $t$. Given such a binary tree $T(\G)$, we use McCormick inequalities to linearize each non-leaf node, and obtain 
\begin{equation*}\label{eq:RMC} \Bigl\{ \z \Bigm| \max \{0,z_{t_\ell} + z_{t_r} -1\} \le z_{t} \le \min \{ z_{t_\ell} , z_{t_r}\} \text{ for each non-leaf node } t \in T(\G)\Bigr\}. \tag{\textsc{RMC}}
\end{equation*}
Note that this formulation may introduce auxiliary variables for non-leaf nodes not in the consideration set $\E$. These variables only track intermediate availability conditions; they do not introduce additional purchase alternatives. The next proposition shows that, when the relaxation is defined by \eqref{eq:RMC}, this resulting perspective formulation is exact.
\begin{proposition}\label{prop:our:exact}
    If the bundle-activation relaxation $\R$ is given by \eqref{eq:RMC}, then the formulation~\eqref{model:pers} exactly solves~\eqref{eq:assortment-single}.
\end{proposition} 

For any fixed recursive decomposition, formulation \eqref{eq:RMC} is straightforward to construct and implement. Consequently, we always include \eqref{eq:RMC} in the relaxation to ensure exactness. 




\textbf{Pairwise Purchase Patterns}\label{sec:odd}
Odd-cycle inequalities apply to consideration sets with singleton products and pairwise bundles. For example, in a bakery two-pack setting, a customer may choose any two types of bread. Products that appear in multiple bundles may create cycles, as illustrated in Figure~\ref{fig:bundle}. Odd-cycle inequalities exploit these shared products to add consistency conditions across overlapping bundle activations, thereby tightening the baseline relaxation in Proposition~\ref{prop:our:exact}. 
We next illustrate this idea using the multi-purchase graph in Figure~\ref{fig:bundle}.
\begin{example}\label{eq:odd-cycle}
In Figure~\ref{fig:bundle}, the edge set $C = \{e_1, e_2,e_3,e_4,e_5\}$ forms a cycle. Then, we select a subset $D$ of the cycle that contains an odd number of edges, for example, $D=\{e_1,e_2,e_3\}$. Recall that for each edge $e = \{i,j\}$ we introduce a variable $z_{e} = z_iz_j$ to indicate whether bundle $\{i,j\}$ is available, and in Section~\ref{sec:rmc}, these bilinear relations are linearized using the McCormick inequalities. In this example,  we consider the following subset of McCormick inequalities:
\begin{align*}
   & - z_{ij}   \le -z_{i}  -z_{j} + 1 \text{ and } -z_{ij} \le  0 && \for \, \{i,j\} \in D\\
   & z_{ij} \le z_i \text{ and }z_{ij} \le z_j && \for \, \{i,j\} \in C\setminus D.
\end{align*} 
Thus, for bundles in the selected subset $D$, the variable $z_e$ is underestimated, whereas for bundles in $C\setminus D$, it is overestimated. Aggregating these inequalities and dividing both sides by $2$ yields 
\[
    z_{e_5} + z_{e_4} - z_{e_1} -z_{e_2} -z_{e_3} \le z_{5} - z_{2} - z_{3} + \frac{\vert D \vert}{2}. 
\]
Since $\vert D \vert$ is odd and $\z$ is binary, the inequality obtained by rounding the constant term $\frac{\vert D \vert}{2}$ down to the nearest integer preserves the validity of the inequality. \hfill \Halmos
\end{example}

The inequality derived in Example~\ref{eq:odd-cycle} is referred to as an \textit{odd-cycle inequality} in~\cite{padberg1989boolean}. Let $C\subseteq \E$ be a cycle and let $ D\subseteq C$ be a subset that has an odd number of bundles. Let $S_{\texttt{in}} = \{u\in V \mid \exists e, f \in D, e\neq f, \st e\cap f = u \}$ be the set of products connecting two distinct bundles in the subset $D$. Similarly, let $S_\texttt{out} = \{u\in V \mid \exists e, f \in C\setminus D, e\neq f, \st e\cap f = u \}$ be the set of products connecting two distinct bundles in $C \setminus D$. The corresponding odd-cycle inequality is given as follows. 
\begin{equation*}\label{eq:odd}
   \sum_{\{i,j\}\in C\setminus D} z_{ij} - \sum_{\{i,j\}\in D}z_{ij} \le  \sum_{i\in S_\texttt{out}} z_i  - \sum_{i\in S_\texttt{in}} z_i + \frac{|D|-1}{2} \tag{\textsc{Odd-cycle}}.
\end{equation*}



We incorporate inequality~\eqref{eq:odd} to strengthen the relaxation $\R$. For series-parallel graphs, they are sufficient to describe the convex hull of the bundle-activation set $\M_{\G}$ when combined with the McCormick constraints~\eqref{eq:RMC}. Therefore, by Theorem~\ref{theorem:main}, the resulting perspective formulation yields an exact LP formulation for ~\eqref{eq:assortment-single}. 
A graph is series-parallel if it arises from a forest by repeatedly replacing edges by parallel edges or by edges in series. The cycle in Figure~\ref{fig:bundle} is a simple example of a series-parallel graph. 

\begin{proposition}\label{prop:odd}
If the bundle-activation relaxation $\R$ is given by \eqref{eq:RMC} and \eqref{eq:odd}, then the formulation~\eqref{model:pers} exactly solves~\eqref{eq:assortment-single} under any multi-purchase hypergraph $\G$. If $\G$ is a series-parallel graph, then its LP relaxation exactly solves unconstrained~\eqref{eq:assortment-single}.
\end{proposition} 


\textbf{Structured Higher-Order Purchase Patterns}\label{sec:ric}
Running-intersection inequalities are useful for higher-order consideration sets whose bundles exhibit nested or chain-like overlaps. Such structures arise when larger bundles are built by extending existing product combinations. 
For example, a customer may consider $\{\textit{bread, cheese}\}$ and $\{\textit{bread, cheese, ham}\}$, where adjacent bundles share products and are therefore not activated independently. Running-intersection inequalities exploit these shared products to impose consistency across nested bundle activations under the same assortment decision.

We formalize this structure through the running intersection property. 
A collection of hyperedges $\tilde{E}$ has the running intersection property if there exists an ordering $e_1,\ldots,e_m$ of $\tilde{E}$ such that, for each $i=2,\ldots,m$, there exists $j<i$ satisfying
\[
e_i \cap (e_1 \cup \cdots \cup e_{i-1}) \subseteq e_j .
\]
Equivalently, each new bundle $e_i$ attaches to the previously listed bundles through an overlap that is contained in a single earlier bundle. 
For ease of exposition, define
\begin{equation}\label{eq:def:n}
    N(e_i) = e_i \cap (e_1 \cup \cdots \cup e_{i-1}), \quad i=2,\ldots,m,
\end{equation}
and set $N(e_1)=\emptyset$. 
In our problem, $N(e_i)$ denotes the shared products through which bundle $e_i$ overlaps with earlier bundles. 
We use the multi-purchase hypergraph in Figure~\ref{fig:ric} to illustrate how the running intersection structure strengthens the baseline relaxation in Proposition~\ref{prop:our:exact}.


\begin{example}\label{eg:derive:ric}
For the hypergraph illustrated in Figure~\ref{fig:ric}, we use the hyperedge $e_0$ as a base bundle given its maximal cardinality. It has three neighbors $e_1,e_2$, and $e_3$, each intersecting with the root $e_0$. We define an induced hyperedge set by the base $e_0$,  $\tilde{\E} = \{e_k\cap e_0 \mid k = 1,2,3  \}$. Note that $\tilde{\E}$ has a running intersection ordering, that is, $e_1 \cap e_0 $, $e_2 \cap e_0 $ and $e_3 \cap e_0 $. To construct the strengthened formulation, for each $k \in \{1,2,3\}$, we define the relaxation of the multilinear variable $z_{e_k}$ based on its neighborhood structure. If $N(e_k \cap e_0)  \neq \emptyset$  then we relax $z_{e_k}$ by $z_{e_k} \leq z_{\nu}$ for some $\nu \in N(e_k \cap e_0)$, and otherwise, we use the trivial bound $z_{e_k} \leq 1$. In our example, this yields the following valid inequalities: 
\begin{equation}\label{eq:eg:ric1}
    z_{e_1} \leq 1, z_{e_2} \leq z_2 \text{ and } z_{e_3} \leq z_3.
\end{equation} 
Furthermore, for node $6 \in e_0 \setminus \cup_{k=1}^3 e_k$, we have the valid bound 
\begin{equation}\label{eq:eg:ric2}
    z_6 \leq 1.
\end{equation} 
Summing these inequalities with the lower bound $0 \leq z_{e_0}$, a naive aggregation results in the valid inequality $z_{e_1} + z_{e_2} + z_{e_3} + z_{6} \le 2 + z_{e_0} + z_{2} + z_{3}$. Now, observe that the running intersection property in $\tilde{\E}$ allows us to show that if $z_{e_0} =0$ then one of the inequalities~\eqref{eq:eg:ric1} and~\eqref{eq:eg:ric2} is not binding. Therefore, we can tighten the aggregated inequality by subtracting one from the right-hand side, resulting in $z_{e_1}+ z_{e_2}+z_{e_3}+z_{6}  \le  1 + z_{e_0} +  z_{2} + z_{3}$.\hfill \Halmos

\end{example}




The inequality derived in Example~\ref{eg:derive:ric} is referred to as a \textit{running-intersection inequality} proposed by~\cite{del2021running}. We now formalize this class of inequalities as follows. Let $e_0$ be a base bundle, and let $e_k$, $k \in K$, be a collection of neighboring bundles such that $\tilde{\E}:=\{e_0\cap e_k \mid k \in K\}$ has the running intersection property. Consider a running intersection ordering of $\tilde{E}$: for each $k \in K$ with $N(e_0 \cap e_k) \neq \emptyset$, let $\mu_k \in N(e_0 \cap e_k)$. A running intersection inequality is defined as:
\begin{equation*}\label{eq:running}
\sum_{k\in K}z_{e_k} + \sum_{v \in e_0 \setminus \bigcup_{k\in K}e_k } z_v  \le  z_{e_0} +  \sum_{k\in K \mid N(e_0 \cap e_k) \neq \emptyset } z_{\mu_k} + \omega-1, \tag{\textsc{RIC}}
\end{equation*} 
where $\omega = \bigl\vert e_0 \setminus \bigcup_{k\in K}e_k\bigr\vert  + \bigl\vert \{k \in K \mid N(e_0 \cap e_k) = \emptyset\} \bigr\vert $, which counts the number of trivial constant bounds used in the derivation.

Adding the running-intersection inequalities~\eqref{eq:running} to the bundle-activation relaxation $\R$ strengthens the recursive McCormick relaxation of the bundle-activation set $\M_{\G}$. 
These inequalities are valid for general consideration sets. More importantly, for a broad class of structured consideration sets, this strengthening is sufficient to fully describe the convex hull of $\M_{\G}$, i.e., $\R$ is integral. 
Specifically, if the multi-purchase hypergraph is $\beta$-acyclic and kite-free, then the recursive McCormick constraints~\eqref{eq:RMC} together with the running-intersection inequalities~\eqref{eq:running} describe $\conv(\M_{\G})$~\citep{del2021running}. 
By Theorem~\ref{theorem:main}, the corresponding perspective formulation therefore yields an exact LP formulation for~\eqref{eq:assortment-single}.

We briefly recall the two structures. A $\beta$-cycle is a sequence $\mu_1,e_1,\mu_2,e_2,\ldots,\mu_m,e_m,\mu_1$, where $m\ge 3$ and each node $\mu_k$ appears only in its two adjacent hyperedges $e_{k-1}$ and $e_k$, with indices taken modulo $m$. 
A kite consists of three hyperedges $e_0,e_1,e_2\in\E$ such that$|e_0 \cap e_1 \cap e_2| \ge 2$, $(e_0\cap e_1) \setminus e_2 \neq \emptyset $, and $(e_0\cap e_2) \setminus e_1 \neq \emptyset$. Intuitively, excluding $\beta$-cycles and kites limits the complexity of overlapping bundles, so that nested activation relationships can be captured by running-intersection inequalities. For example, the hypergraph in Figure~\ref{fig:ric} is both $\beta$-acyclic and kite-free. The corresponding result is formalized in the following proposition.

\begin{proposition}\label{prop:exact:running}
    If the bundle-activation relaxation $\R$ is given by \eqref{eq:RMC} and \eqref{eq:running},  the formulation~\eqref{model:pers} exactly solves~\eqref{eq:assortment-single} under any multi-purchase hypergraph $\G$. If  $\G$ is a kite-free $\beta$-acyclic hypergraph, then its LP relaxation exactly solves unconstrained~\eqref{eq:assortment-single}.
\end{proposition}

Our LP results also connect our formulation to the literature on tractable LP characterizations for assortment optimization. \cite{lo2019assortment} study assortment optimization with product synergy effects, represent pairwise interactions by a graph, and show that an exact LP formulation exists when the interaction graph is a path. 
Our results extend this line of work in two directions. 
For pairwise-bundle consideration sets, odd-cycle inequalities yield LP formulations for series-parallel graphs, going beyond path structures. 
For higher-order consideration sets, running-intersection inequalities yield LPs for kite-free $\beta$-acyclic hypergraphs.  

This extension is important because existing path-based LP formulations and tree-decomposition-based dynamic programs do not directly provide LP characterizations for the assortment optimization problem under Logit-MP. 
Our framework shows that, for the unconstrained Logit-MP assortment problem, how purchase bundles overlap determines whether the problem is solvable through LPs. The result provides a unified LP-based route for structured multi-purchase assortment optimization.

\subsection{Benchmark Formulations and Strength Comparison}\label{sec:naive:model} 
In this section, we present an alternative class of formulations that do not rely on perspective reformulations. Using the Charnes--Cooper transformation \citep{charnes1962programming}, the choice probability variables $\rho$ and $\y$ can be equivalently expressed as:
\[ 
  \rho +\sum_{e\in \E}v_e y_e =1 ,\ \rho \ge 0,\  \z \in \M_{\G},\ y_e =   \rho z_e \, \for e \in \E  .
\] 
This representation introduces bilinear terms of the form $y_e = \rho z_e$. Replacing $\M_{\G}$ by a relaxation $\R$ and linearizing the bilinear products using McCormick inequalities yields 
\begin{subequations}\label{eq:bigm} 
    \begin{alignat}{3}
    \max_{\rho,\x,\y,\z} \quad & \sum_{e\in \E } r_e v_e y_e  \tag{\textsc{Big-M}} \label{model:bigm} \\
   \st  &\rho +\sum_{e\in \E}v_e y_e =1 \quad \rho \ge 0  \quad \z \in \R  \quad \text{and }  \x \in \{0,1\}^{\vert V\vert} \label{eq:bigm:0}\\
    &\max\{  \rho_L z_e,\ \rho_U z_e + \rho  - \rho_U \} \leq  y_e  \le  \min\{  \rho_L z_e + \rho - \rho_L,\   \rho_U z_e \}  \quad  && \for e \in \E. \label{eq:bigm-1}  
    \end{alignat}
\end{subequations}  
The inequalities in \eqref{eq:bigm-1} are obtained by using McCormick inequalities with a lower bound $\rho_L$ and upper bound $\rho_U$ on the no-purchase probability $\rho$. The strength of formulation~\eqref{eq:bigm} depends on these bounds. We therefore refer to \eqref{eq:bigm} as the \emph{Big-M} formulation. This formulation preserves the exactness of the $\R$, establishing the following result.
 
\begin{proposition}\label{prop:bigm:exact}
  If the bundle-activation relaxation $\R$ is exact, then the formulation~\eqref{model:bigm} is an MIP formulation for the assortment optimization problem~\eqref{eq:assortment-single}.
\end{proposition}

In contrast to the perspective formulation, the Big-M formulation does not, in general, preserve the polyhedron strength of the bundle-activation relaxation. That is, even if $\R$ is exact and integral, the LP relaxation of formulation~\eqref{model:bigm} cannot solve the assortment problem exactly. The following example illustrates this gap. 
\begin{example}
    Consider a simple graph $\G(V,\E)$, where $V=\{1,2\}$ and $\E = \{(1) ,(2), (1,2)\}$. Suppose $v_1 =v_2=v_{12}=3$, $r_1=r_2=1$, and $r_{12} = r_1+r_2=2$. Since $\G$ is a series-parallel graph, the LP relaxation of formulation~\eqref{model:pers} exactly solves the unconstrained Logit-MP assortment problem by Proposition~\ref{prop:odd}. The optimal solution is
    \[ \x=(1,1)^T,\rho = 0.1, \y=(0.3,0.3,0.3)^T,\]
    with revenue 1.2.
    However, the LP relaxation of~\eqref{model:bigm} contains a fractional solution with revenue larger than 1.2. Specifically, we set the bounds of the no-purchase variable $\rho$ as $\rho_U = 1$ and $\rho_{L}=1/(1+3+3+3)=1/10$. A feasible solution of the LP relaxation of~\eqref{model:bigm} is 
    \[ \x=(1/5,1/5)^T, \rho = 7/25,\y =(1/50,1/50,1/5)^T,\z = (1/5,1/5,1/5)^T,\]
with revenue $33/25>1.2$.  \hfill \Halmos
\end{example} 

We next formally compare the strength of the two classes of formulations. A key measure of the quality of an MIP formulation is the tightness of its LP relaxation, obtained by dropping the integrality constraints. This criterion is particularly relevant in practice, since tighter relaxations typically lead to stronger bounds and faster convergence of the branch-and-bound algorithms underlying modern commercial MIP solvers. Since the optimization problem has a maximization objective, a tighter formulation has a smaller LP relaxation objective value. Let $Z_{\texttt{pers}}$ and $Z_{\texttt{Big-M}}$ be the optimal objective value of the LP relaxation of~\eqref{model:pers} and~\eqref{model:bigm}, respectively. The two objective values have the following relationship.

\begin{theorem}\label{thm:compare}
For any relaxation $\R$, $Z_{\texttt{pers}} \le Z_{\texttt{Big-M}}$.  
\end{theorem}

Note that, relative to the perspective formulation, the Big-M formulation introduces additional variables $\z$. An advantage of explicitly modeling $\z$ is that it enables us to use conic inequalities that further strengthen formulation~\eqref{eq:bigm}.
\begin{remark}
Following \cite{sen2018conic}, we augment~\eqref{model:bigm} with the following second-order cone constraints:
\begin{subequations}\label{eq:conic}
    \begin{alignat}{3}
        &\rho\bigl(1 + \sum_{e\in \E}v_ez_e\bigr)\geq 1,  \label{eqc:1}\\
      &y_e \bigl(1 + \sum_{e\in \E}v_ez_e\bigr)\geq z_e^2,\quad \forall e \in \E. \label{eqc:2}  
    \end{alignat}
\end{subequations} 
The first inequality follows from the definition of the no-purchase variable $\rho$. The second inequality leverages the relationship $y_e \ge \rho z_e$ together with the binary identity $z_e = z_e^2$. \hfill \Halmos
\end{remark}



\section{Numerical Studies on Real Data}\label{sec:estimation} 

This section presents an empirical study using bakery transaction data and semi-synthetic data generated from reported preference rankings. We estimate sparse multi-purchase hypergraphs and evaluate their implications for prediction and assortment optimization. The results show that explicitly modeling multi-purchase behavior improves predictive performance and can generate meaningful revenue improvements for assortment decisions.


\subsection{Data and Multi-Purchase Hypergraph Construction}\label{sec:est:method}

We use transaction data from a bakery~\citep{bakerydata}. 
Each transaction records the purchase time, the purchased bundle, and product prices. 
We focus on August 2022, the month with the largest number of transactions. 
The resulting dataset contains 59 products and 11,507 transactions. Of these transactions, 56.8\% contain one product, 24.1\% contain two products, 13.2\% contain three products, and 5.9\% contain four or more products. Thus, multi-purchase transactions account for more than 40\% of the data and represent an important feature of demand in this setting.

A natural starting point for constructing the multi-purchase hypergraph is to treat all observed purchase patterns as candidate bundles. The data contain 1,552 distinct multi-item bundles, which we denote by $\E_{\mathrm{data}}$. Together with the singleton bundles, they define the empirical hypergraph $\G(V,V\cup \E_{\mathrm{data}})$. Under a mild support condition, this empirical support recovers the true consideration set with probability approaching one exponentially fast in the sample size; see~\ref{ec:support:recovery}.

The empirical support is large but highly concentrated. The top 20\% of multi-item bundle types account for 72.68\% of all multi-purchase transactions. This concentration motivates us to construct sparse hypergraphs that prioritize frequently purchased bundles.
Starting with $\E_{\mathrm{data}}$, we vary a sparsity level $\theta$ to construct a sequence of hypergraphs by retaining the most frequently purchased multi-item bundles. Specifically, for each bundle $e\in \E_{\mathrm{data}}$, let $N_e$ denote its total transaction frequency observed in the dataset. Given a sparsity level $\theta \in [0,1]$, we retain the $m_{\theta} = \lfloor \theta |\E_{\mathrm{data}}| \rfloor$ most frequently purchased bundles, denoted by $\E_{\theta}$.  
Varying $\theta$ generates a family of candidate hypergraphs $\G_\theta=(V,V \cup \E_\theta)$. $\theta$ ranges from the singleton MNL support ($\theta=0$) to the full empirical support ($\theta=1$).




For each candidate hypergraph $\G_\theta$, we estimate the utility parameters by maximum likelihood under the full and pairwise utility specifications introduced in Section~\ref{sec:logit:model}. When an observed purchase bundle is not retained in $\E_{\mathrm{data}}$, we enumerate all feasible partitions of that bundle and whose elements all belong to $\E_{\theta}$. Following the splitting idea in~\cite{tulabandhula2023multi}, we assign equal weight to each feasible partition and use the resulting weighted observations for estimation. Estimation details are provided in~\ref{ec:estimation:details}. We evaluate the resulting candidate models using five-fold out-of-sample prediction in the next subsection.

\subsection{Estimation Results}\label{sec:est:result} 
Having constructed a family of candidate hypergraphs, we next evaluate whether retaining the long tail of observed bundles improves out-of-sample prediction. Figure~\ref{fig:predict} reports five-fold test log-likelihood and marginal root mean square error (MRMSE). MRMSE is a standard predictive performance metric in choice modeling literature~\cite[e.g.,][]{berbeglia2022comparative,lin2026intra}, computed as 
\[ \texttt{MRMSE} = \sqrt{\frac{1}{\sum_{t\in \mathcal{T}}|S_t|}\sum_{t\in \mathcal{T}}\sum_{j\in S_t} \bigl(I(j,S_t) - P(j,S_t) \bigr)^2 }, \]
where $\mathcal{T}$ denotes all transactions, $I(j,S_t)$ is an indicator denoting whether $j$ is purchased in transaction $t$ with assortment $S_t$, and $P(j,S_t)=\sum_{e\in \E_{\theta},e\ni j}\pr(e|S_t)$ is the marginal probability of purchasing $j$.


Figure~\ref{fig:predict} shows that most of the predictive gains are obtained at relatively small sparsity levels. As $\theta$ increases from zero, both predictive metrics improve rapidly, but the marginal gains diminish thereafter. In particular, the split-weight log-likelihood begins to deteriorate beyond $\theta=0.05$, with the decline being more pronounced under the Pairwise specification, possibly due to the pairwise structure imposed on the utility function. Meanwhile, at $\theta=0.1$, the reduction in MRMSE relative to the single-purchase baseline $\theta=0$ already accounts for about 58\% of the total reduction from $\theta=0$ to $\theta=1$. We therefore use $\theta=0.1$ in the subsequent model comparisons. These results reveal a sparsity--prediction tradeoff: a sparse hypergraph captures a substantial share of the predictive gains, while denser hypergraphs can further improve marginal prediction at the cost of enlarging the downstream assortment optimization problem.

\begin{figure}[htbp]
   \FIGURE
{\includegraphics[width=0.90\linewidth]{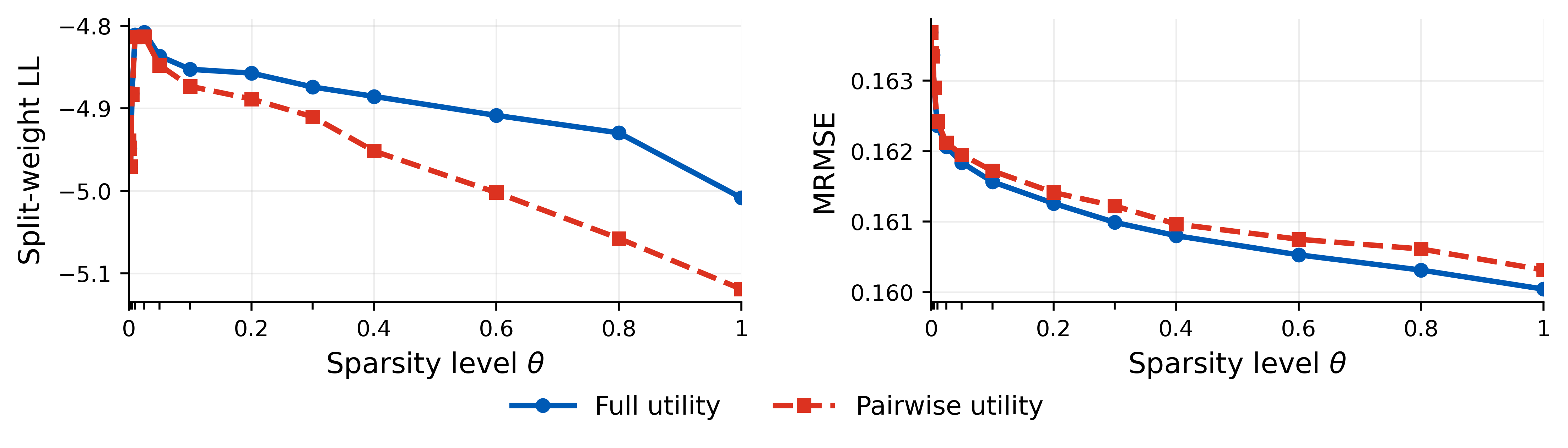}}
   {Predictive performance under different sparsity levels\label{fig:predict}}
  {Results are averaged over five folds; higher split-weight log-likelihood and lower MRMSE indicate better predictive performance.}
\end{figure}

We next compare the selected sparse Logit-MP models with MNL, Mixed-MNL (MMNL), and BundleMVL-2. Following \citet{tulabandhula2023multi}, multi-item bundles are transformed into singleton observations for MNL/MMNL and size-at-most-two observations for BundleMVL-2. Table~\ref{tab:em-fullday-test} reports model size and out-of-sample performance. The 4-class MMNL specification is selected by MRMSE from $k\in \{2,3,4\}$, while $\theta=0.1$ is selected from Figure~\ref{fig:predict}.

\begin{table}[htbp]
\TABLE
{Estimation performance of different choice models\label{tab:em-fullday-test}} 
{\begin{tabular}{lccc}
\toprule
Model & Selected specification & Avg. \# Parameters & MRMSE (Improv. vs. MNL)   \\
\midrule
MNL & -- & 59 & 0.1643 (--)   \\
MMNL & $k=4$ & 239 & 0.1643 (0.01\%)  \\
BundleMVL-2 & -- & 1770 & 0.1614 (1.78\%)  \\
Sparse-Full & $\theta=0.1$ & 191.4 & 0.1616 (1.65\%)   \\
Sparse-Pairwise & $\theta=0.1$ & 182.2 & 0.1617 (1.56\%)  \\
\bottomrule
\end{tabular}}
{Results are averaged over five folds. The MMNL specification $k=4$ is selected by MRMSE, while $\theta=0.1$ is selected from the sparsity-prediction frontier in Figure~\ref{fig:predict}.}
\end{table}


The comparison yields two main findings. First, all explicit multi-purchase models improve MRMSE relative to MNL, whereas the 4-class MMNL provides little improvement in this dataset. Thus, in this setting, adding heterogeneity to a single-purchase model does not recover the predictive gains from explicitly modeling multi-purchase behavior. Second, Sparse-Full and Sparse-Pairwise achieve similar MRMSE to BundleMVL-2  while using about 89\% fewer parameters. 
These results indicate that the value of a general hypergraph lies in the data-driven flexibility: it allows the bundle support to adapt to the data, achieving competitive prediction with a small model while retaining the option to trade additional complexity for further predictive gains.

\subsection{Additional Numerical Experiment Based on the Sushi Dataset}\label{sec:sushi}
The bakery dataset provides valuable real-world transaction data, but it also comes with limitations. In particular, it does not provide a ground-truth choice model for comparing assortment decisions under different choice models. To address this problem, we conduct a semi-synthetic experiment by combining the bakery dataset and the sushi dataset~\citep{kamishima2003nantonac}. 

We construct a ground-truth choice model via the multi-choice ranked-list mechanism in \cite{lin2026intra}. The sushi dataset records 5000 customers' preference rankings over 10 sushi types. Following the Top-3 specification rule in the literature~\cite[e.g.][]{berbeglia2022comparative,akchen2021assortment,costa2026modern}, we assume that each customer considers only her three most preferred products. Given an assortment, a customer will purchase the $K$ highest-ranked available products, where $K$ is calibrated from the bakery bundle-size distribution. If none of the three most preferred products are available, the customer buys nothing.
To examine the effects of no-purchase rate, following~\cite{akchen2021assortment}, we additionally introduce a no-purchase customer type and vary its arrival rate. Motivated by the bakery application, our experiments consider relatively high purchase-incidence regimes, with the no-purchase ratio varied to examine robustness across these regimes. The detailed data generation process is provided in~\ref{ec:sushi:generation}. 

Given each generated dataset, we estimate the candidate choice models following the procedure in Section~\ref{sec:est:result} using five-fold cross-validation. We select the number of customer segments for MMNL based on out-of-sample MRMSE and select the sparsity level $\theta$ for each Logit-MP specification from its prediction–sparsity frontier. We then optimize the assortment under each selected model. We use our perspective formulation to compute optimal assortments under the homogeneous MNL, BundleMVL-2, and sparse Logit-MP models, and conic formulations in~\cite{sen2018conic} under the MMNL model. Because the data-generating process is known, we evaluate each resulting assortment by directly computing its expected revenue under the ground-truth choice mechanism described above.

\begin{table}[htbp]
\TABLE
{Revenue ratios and assortment size under different no-buy ratios and cardinality constraints\label{tab:sushi:lift}}
{
 \resizebox{\textwidth}{!}{%
\begin{tabular}{cccc*{5}{c}*{3}{c}}
\toprule
& & & & \multicolumn{5}{c}{Revenue ratio to true optimum (offer size)} & \multicolumn{3}{c}{Improvement vs. MNL} \\
\cmidrule(lr){5-9} \cmidrule(lr){10-12}
\shortstack{No-buy\\customer ratio} & \shortstack{Realized no-buy\\transaction ratio} & Cardinality & \shortstack{True optimal\\size} & MNL & 4-MMNL & BundleMVL-2 & Full  & Pairwise & BundleMVL-2 & Full & Pairwise \\
\midrule
0.05 & 0.189 & 3 & 3 & 0.929 (3) & 0.929 (3) & 0.929 (3) & 0.998 (3) & 0.998 (3) & 0.0\% & 7.4\% & 7.4\% \\
0.05 & 0.189 & 5 & 5 & 0.731 (3) & 0.731 (3) & 0.749 (4) & 0.995 (5) & 0.995 (5) & 2.5\% & 36.2\% & 36.2\% \\
0.05 & 0.189 & 7 & 7 & 0.677 (3) & 0.677 (3) & 0.694 (4) & 0.982 (6.2) & 0.982 (6.2) & 2.5\% & 45.1\% & 45.1\% \\
\midrule
0.10 & 0.231 & 3 & 3 & 0.929 (3) & 0.929 (3) & 0.929 (3) & 0.996 (3) & 0.996 (3) & 0.0\% & 7.1\% & 7.1\% \\
0.10 & 0.231 & 5 & 5 & 0.731 (3) & 0.731 (3) & 0.824 (4.6) & 0.983 (5) & 0.983 (5) & 12.8\% & 34.6\% & 34.6\% \\
0.10 & 0.231 & 7 & 7 & 0.677 (3) & 0.677 (3) & 0.830 (5.2) & 0.987 (6.4) & 0.987 (6.4) & 22.6\% & 45.7\% & 45.7\% \\
\midrule
0.15 & 0.271 & 3 & 3 & 0.929 (3) & 0.929 (3) & 0.929 (3) & 0.994 (3) & 0.994 (3) & 0.0\% & 6.9\% & 6.9\% \\
0.15 & 0.271 & 5 & 5 & 0.731 (3) & 0.731 (3) & 0.915 (5) & 0.977 (5) & 0.977 (5) & 25.3\% & 33.7\% & 33.7\% \\
0.15 & 0.271 & 7 & 7 & 0.677 (3) & 0.677 (3) & 0.987 (7) & 1.000 (7) & 1.000 (7) & 45.7\% & 47.7\% & 47.7\% \\
\midrule
0.20 & 0.314 & 3 & 3 & 0.929 (3) & 0.929 (3) & 0.972 (3) & 0.996 (3) & 0.996 (3) & 4.6\% & 7.1\% & 7.1\% \\
0.20 & 0.314 & 5 & 5 & 0.858 (4) & 0.782 (3.4) & 0.977 (5) & 1.000 (5) & 1.000 (5) & 13.9\% & 16.6\% & 16.6\% \\
0.20 & 0.314 & 7 & 7 & 0.795 (4) & 0.724 (3.4) & 0.987 (7) & 0.956 (7) & 0.956 (7) & 24.1\% & 20.2\% & 20.2\% \\
\bottomrule
\end{tabular}%
}
}{Each entry is averaged over five independent datasets. Revenue is evaluated under the known Sushi semi-synthetic data-generating process. The Logit-MP results use $\theta=0.2$. The average numbers of parameters across the 20 datasets are 10 for MNL, 43 for 4-MMNL, 55 for BundleMVL-2, and 29.6 for both Sparse-Full and Sparse-Pairwise Logit-MP with $\theta=0.2$.}
\end{table} 

Table~\ref{tab:sushi:lift} reports the revenue achieved by each estimated model as a fraction of the true optimal revenue under the known rank-based data-generating process. The prediction-sparsity frontier selects $\theta=0.2$ for both sparse Logit-MP specifications in the reported experiments. The result shows that all multi-purchase models remain effective although the ground-truth choice mechanism is rank-based rather than logit-based. The sparse Logit-MP models under both the Full and Pairwise utility specifications remain robust, attaining between 95.6\% and 100\% of the true optimal revenue across all settings. Their advantage over BundleMVL-2 is more pronounced when the no-buy customer ratio is low; for example, at a ratio of 0.05, Full and Pairwise achieve 98.2--99.8\% of the true optimum, whereas BundleMVL-2 achieves 69.4--92.9\%. This pattern suggests that, when purchase data are sufficiently informative, a sparse bundle support adapted to the observed demand can still yield near-optimal assortment decisions.

Second, explicitly modeling multi-purchase behavior leads to substantial gains over the single-purchase benchmarks, particularly when the assortment constraint is less restrictive. Averaged over the 12 settings, Full and Pairwise improve revenue over MNL by about 25.7\%, compared with 12.8\% for BundleMVL-2. The improvement increases when more products can be offered. A possible reason is that as capacity expands, the multi-purchase models tend to offer more products and thereby create additional bundle-purchase opportunities, whereas MNL often leaves part of the capacity unused. In contrast, allowing heterogeneity through 4-MMNL yields little improvement over MNL. Thus, in this experiment, accounting for preference heterogeneity within a single-purchase model does not substitute for explicitly modeling multi-purchase behavior.

\section{Numerical Studies on Synthetic Data }\label{sec:app}  
In this section, we use the synthetic experiments to answer three questions: how the perspective formulation compares with standard alternatives, when structure-specific inequalities provide additional value, and how the complexity of multi-purchase patterns affects scalability. 
\subsection{Experimental Design}\label{sec:numerical-1}
We first describe the benchmark formulations, instance-generation scheme, and performance metrics used throughout the experiments. 
We compare the perspective formulation~\eqref{model:pers} with the Big-M~\eqref{model:bigm} and conic benchmarks~\eqref{eq:conic}, using recursive McCormick inequalities~\eqref{eq:RMC} as the baseline and odd-cycle~\eqref{eq:odd} or running-intersection inequalities~\eqref{eq:running} for strengthening.
All tests were implemented via Julia~\citep{bezanson2017julia}, using the modeling language JuMP~\citep{Lubin2023}, on a machine with an Intel(R) Core(TM) Ultra 7 265K CPU at 3.90\,GHz and 48\,GB RAM. All models were solved via Gurobi~12.0.3~\citep{gurobi}.

Implementing these formulations requires a separation-based approach. In the perspective formulation~\eqref{model:pers}, constraints~\eqref{eq:pers:oracle-4} and~\eqref{eq:pers:oracle-5} involve maximization and minimization over subsets of $\E$. 
Writing them explicitly leads to $\mathcal{O}(|V|2^{|\E|})$ constraints. 
In addition, the strengthened relaxation requires identifying cycles and running-intersection structures in the multi-purchase hypergraph $\G$. Since a dense or highly overlapping hypergraph may induce a large family of inequalities, adding them all is impractical. We therefore develop a cutting-plane algorithm that avoids enumeration of the full constraint set (see details in~\ref{ec:sec:algo}).




To generate multi-purchase bundles, we consider three design factors: bundle size, hypergraph sparsity, and category composition. Motivated by the case-study observation that most transactions contain no more than three products, we first set the maximum bundle size as $d\in \{2,3\}$. Let $\mathcal{B}_d=\{e\subseteq V:2\le |e|\le d\}$ denote the pool of all candidate bundles. 
We then randomly partition products into three categories and classify each candidate bundle as either cross-category or intra-category. Let $\mathcal{B}_d^{\mathrm{cross}}$ and $\mathcal{B}_d^{\mathrm{intra}}$ denote the corresponding two pools. This categorization is motivated by category-based Logit-MP models such as MVMNL. We use two parameters to control the structure of the resulting hypergraph. The sparsity parameter $\delta$ determines the total number of retained multi-item bundles, $m_\delta=\lfloor \delta |\mathcal B_d|\rfloor$,
while the composition parameter $\pi$ determines the fraction drawn from the cross-category pool. Specifically, we sample $m_{\mathrm{cross}}=\lfloor \pi m_\delta\rfloor $ and $
m_{\mathrm{intra}}=m_\delta-m_{\mathrm{cross}}$ 
bundles uniformly without replacement from $\mathcal B_d^{\mathrm{cross}}$ and $\mathcal B_d^{\mathrm{intra}}$, respectively. Let $\mathcal B^{\mathrm{random}}$ denote the resulting sampled bundles and the consideration set is $\E = V\cup \mathcal B^{\mathrm{random}}$.



We use the function~\eqref {eq:u:pairwise} to assign an attraction value for each bundle. Specifically, we draw a latent value $\alpha_{i}$ from a uniform distribution $\mathcal{U}(-1.5,0.5)$ for each $i \in V$ and $\beta_{ij} \sim \mathcal{U}(-1.25,0.75)$. The ranges approximately cover the empirical percentiles of the corresponding bakery estimates. The single-product revenue $r_i$'s are drawn independently from an i.i.d. distribution $\mathcal{U}(1, 4)$ and the bundle revenue is additive, that is, $r_e = \sum_{i \in e} r_i$. We consider a cardinality constraint $\sum_{i \in V} x_i \leq 0.2  |V| $. For each parameter configuration, we generate 10 random instances following the aforementioned procedure. All experiments are conducted with a time limit of 3,600 seconds and a relative optimality tolerance of 0.05\%. 

We consider four metrics to compare computational performance: $\#\text{sol}$ is the number of instances solved to global optimality within the time limit; $\text{Time}$ is the average computation time in seconds; $\text{Gap}$ denotes the average optimality gap at termination, defined as $\text{Gap} = 100\% \times \frac{Z_U - Z^*}{Z^*}$, where $Z_U$ is the revenue of the best upper bound and $Z^*$ is the revenue of the best integer solution obtained at termination via the perspective formulation; $\#\text{node}$ is the average number of branch-and-bound nodes explored in the search tree. It is also helpful to compare the natural continuous relaxation of different formulations. We compute the root gap by $\text{RGap} = 100\% \times \frac{Z_{\texttt{Rlx}} - Z^*}{Z^*}$, where $Z_{\texttt{Rlx}}$ denotes the revenue of the continuous relaxation of the given formulation.

\begin{table}[htbp]
\TABLE 
{{Computation results for instances of pure intra-category bundles with maximum bundle size two}\label{oc_1}}
{\begin{tabular}{c l rrrrr rrrrr}
\toprule
\multicolumn{1}{c}{Instance} & \multirow{2}{*}{Model} 
& \multicolumn{5}{c}{Recursive McCormick} 
& \multicolumn{5}{c}{RMC \& Odd Cycle} \\
\cmidrule(r){1-1} \cmidrule(lr){3-7} \cmidrule(l){8-12}
$N$ & & \#sol & Time(s) & Gap(\%) & RGap(\%) & \#node & \#sol & Time(s) & Gap(\%) & RGap(\%) & \#node \\
\midrule
\multirow{3}{*}{100} 
& Pers & 10 & 0.1 (0.1) & 0.0 & 0.0 & 1 & 10 & 0.1 (0.1) & 0.0 & 0.0 & 1 \\
& Big-M & 10 & 2.9 (0.0) & 0.0 & 12.6 & 9737 & 10 & 3.4 (0.1) & 0.0 & 12.6 & 11705 \\
& CONIC & 10 & 7.7 (0.0) & 0.0 & 10.8 & 5405 & 10 & 8.9 (0.9) & 0.0 & 10.8 & 7841 \\
\midrule
\multirow{3}{*}{500} 
& Pers & 10 & 10.1 (5.2) & 0.0 & 0.0 & 1 & 10 & 10.0 (5.4) & 0.0 & 0.0 & 1 \\
& Big-M & 2 & 1622.9 (0.0) & 2.8 & 9.3 & 193394 & 3 & 1535.7 (4.0) & 2.4 & 9.3 & 164121 \\
& CONIC & 0 & - (0.0) & 7.9 & 9.0 & 10063 & 0 & - (3172.1) & 8.5 & 9.9 & 13645 \\
\midrule
\multirow{3}{*}{1000} 
& Pers & 10 & 324.6 (60.5) & 0.0 & 0.0 & 1 & 10 & 329.5 (63.1) & 0.0 & 0.0 & 1 \\
& Big-M & 0 & - (0.0) & 9.0 & 11.0 & 103476 & 0 & - (17.7) & 8.8 & 11.0 & 105271 \\
& CONIC & 0 & $\smallsetminus$ & $\smallsetminus$ & $\smallsetminus$ & $\smallsetminus$ & 0 & $\smallsetminus$ & $\smallsetminus$ & $\smallsetminus$ & $\smallsetminus$ \\
\bottomrule
\end{tabular}}
{``-" in the Time(s) column indicates that the corresponding model fails to solve the MIP within 3600 seconds for all instances. ``$\smallsetminus$" for all metrics indicates that the corresponding model fails to solve the linear relaxation within 3600 seconds for all instances. The number in parentheses denotes the separation time. We fix $\delta=0.25$ and $\pi=0$.}
\end{table}

\begin{table}[htbp]
\TABLE 
{{Computation results for instances of pure cross-category bundles with maximum bundle size two}\label{oc_2}}
{\begin{tabular}{c l rrrrr rrrrr}
\toprule
\multicolumn{1}{c}{Instance} & \multirow{2}{*}{Model} 
& \multicolumn{5}{c}{Recursive McCormick} 
& \multicolumn{5}{c}{RMC \& Odd Cycle} \\
\cmidrule(r){1-1} \cmidrule(lr){3-7} \cmidrule(l){8-12}
$N$ & & \#sol & Time(s) & Gap(\%) & RGap(\%) & \#node & \#sol & Time(s) & Gap(\%) & RGap(\%) & \#node \\
\midrule
\multirow{3}{*}{100} 
 & Pers & 10 & 0.2 (0.2) & 0.0 & 0.0 & 1 & 10 & 0.2 (0.2) & 0.0 & 0.0 & 1 \\
 & Big-M & 10 & 33.1 (0.0) & 0.0 & 11.0 & 34228 & 10 & 28.9 (0.1) & 0.0 & 11.0 & 30762 \\
 & CONIC & 10 & 238.5 (0.0) & 0.0 & 10.0 & 50232 & 10 & 216.3 (147.4) & 0.0 & 10.0 & 42329 \\
\midrule
\multirow{3}{*}{500} 
& Pers & 10 & 83.4 (22.0) & 0.0 & 0.2 & 1 & 10 & 83.3 (22.4) & 0.0 & 0.2 & 1 \\
& Big-M & 0 & - (0.0) & 3.3 & 6.9 & 106343 & 0 & - (16.8) & 3.4 & 6.9 & 107076 \\
& CONIC & 0 & $\smallsetminus$ & $\smallsetminus$ & $\smallsetminus$ & $\smallsetminus$ & 0 & $\smallsetminus$ & $\smallsetminus$ & $\smallsetminus$ & $\smallsetminus$ \\
\midrule
\multirow{3}{*}{1000} 
& Pers & 10 & 724.1 (269.6) & 0.0 & 0.0 & 1 & 10 & 725.4 (269.7) & 0.0 & 0.0 & 1 \\
& Big-M & 0 & - (0.0) & 5.1 & 6.6 & 32524 & 0 & - (29.1) & 5.1 & 6.6 & 32481 \\
& CONIC & 0 & $\smallsetminus$ & $\smallsetminus$ & $\smallsetminus$ & $\smallsetminus$ & 0 & $\smallsetminus$ & $\smallsetminus$ & $\smallsetminus$ & $\smallsetminus$ \\
\bottomrule
\end{tabular}}
{``-" in the Time(s) column indicates that the corresponding model fails to solve the MIP within 3600 seconds for all instances. ``$\smallsetminus$" for all metrics indicates that the corresponding model fails to solve the linear relaxation within 3600 seconds for all instances. The number in parentheses denotes the separation time. We fix $\delta=0.25$ and $\pi=1$.}
\end{table}

\subsection{Pairwise Purchase Patterns}\label{sec:numerical-2}
We first consider pairwise purchase patterns $d=2$, where bundle activation is bilinear. We use recursive McCormick relaxation and odd-cycle inequalities to construct the relaxation $\R$. 

Table~\ref{oc_1} and~\ref{oc_2} report computational results for instances with pure intra-category and cross-category bundles, respectively. Both tables show that the perspective formulation dominates across all instances. For small instances $N=100$, all three models reach optimality, but our formulation is substantially faster, requiring less than 0.2 seconds, with near-zero root gaps and a single search node. The advantage grows sharply with $N$: for $N=500$ and $1000$, the perspective formulation solves all instances, while two benchmark formulations terminate with large remaining gaps. The results also show that in these pairwise instances, adding odd-cycle inequalities leads to only minor differences in computational performance.

The comparison between the two tables shows that intra-category instances are computationally easier than cross-category instances. 
This is because intra-category bundles induce a nearly decomposable graph, whereas cross-category bundles create stronger coupling across categories.

\subsection{Higher-Order Purchase Patterns}\label{sec:numerica-3}
We next consider higher-order purchase patterns with $d=3$, which introduce substantially richer overlap among bundle-activation terms. We add running-intersection inequalities to tighten the linear relaxation $\R$. 


Table~\ref{tab:ric_1} and~\ref{tab:ric_2} report computational results for instances with pure intra-category and cross-category bundles, respectively. The perspective formulation again uniformly dominates. In contrast, the Big-M formulation fails to solve any instance to optimality within the time limit, while the conic formulation even fails to solve the root node relaxation. 

A key additional observation from both tables is the strength of running-intersection inequalities. Incorporating running-intersection inequalities yields computational gains, which increase the number of instances solved to optimality. In particular, for cross-category instances with $N=140$, the perspective formulation, equipped only with the recursive McCormick relaxation, fails to solve two instances. Augmenting the formulation with running-intersection inequalities resolves all instances and delivers speedups of up to a factor of five, even when accounting for the separation time. These results demonstrate that for $d=3$, efficiently solving higher-order instances benefits from both the perspective formulation and structure-specific strengthening through running-intersection inequalities~\eqref{eq:running}.
\begin{table}[htbp]
\TABLE 
{{Computation results for instances of pure intra-category bundles with maximum bundle size three}\label{tab:ric_1}}
{\begin{tabular}{c l rrrrr rrrrr}
\toprule
\multicolumn{1}{c}{Instance} & \multirow{2}{*}{Model} 
& \multicolumn{5}{c}{Recursive McCormick} 
& \multicolumn{5}{c}{RMC \& Running-Intersection} \\
\cmidrule(r){1-1} \cmidrule(lr){3-7} \cmidrule(l){8-12}
$N$ & & \#sol & Time(s) & Gap(\%) & RGap(\%) & \#node & \#sol & Time(s) & Gap(\%) & RGap(\%) & \#node \\
\midrule

\multirow{3}{*}{180}   
& Pers & 10 & 100.9 (5.5) & 0.0 & 1.9 & 2708 & 10 & 16.9 (41.7) & 0.0 & 0.5 & 78 \\
& Big-M & 0 & - (0.0) & 4.6 & 12.4 & 161636 & 0 & - (22.4) & 4.6 & 12.4 & 159971 \\
& CONIC & 0 & $\smallsetminus$ & $\smallsetminus$ & $\smallsetminus$ & $\smallsetminus$ & 0 & $\smallsetminus$ & $\smallsetminus$ & $\smallsetminus$ & $\smallsetminus$ \\
\midrule
\multirow{3}{*}{220} 
& Pers & 10 & 471.4 (16.3) & 0.0 & 1.9 & 5619 & 10 & 58.4 (116.9) & 0.0 & 0.3 & 86 \\
& Big-M & 0 & - (0.0) & 6.3 & 12.6 & 80805 & 0 & - (11.3) & 6.3 & 12.6 & 79613 \\
& CONIC & 0 & $\smallsetminus$ & $\smallsetminus$ & $\smallsetminus$ & $\smallsetminus$ & 0 & $\smallsetminus$ & $\smallsetminus$ & $\smallsetminus$ & $\smallsetminus$ \\
\midrule
\multirow{3}{*}{260}   
& Pers & 9 & 914.6 (31.3) & 0.6 & 1.9 & 4933 & 10 & 308.3 (266.7) & 0.0 & 0.4 & 629 \\
& Big-M & 0 & - (0.0) & 4.8 & 9.9 & 37087 & 0 & - (14.4) & 4.9 & 9.9 & 38646 \\
& CONIC & 0 & $\smallsetminus$ & $\smallsetminus$ & $\smallsetminus$ & $\smallsetminus$ & 0 & $\smallsetminus$ & $\smallsetminus$ & $\smallsetminus$ & $\smallsetminus$ \\
\bottomrule
\end{tabular}}
{``-" in the Time(s) column indicates that the corresponding model fails to solve the MIP within 3600 seconds for all instances. ``$\smallsetminus$" for all metrics indicates that the corresponding model fails to solve the linear relaxation within 3600 seconds for all instances. The number in parentheses denotes the separation time. We fix $\delta=0.25$ and $\pi=0$.}
\end{table}

\begin{table}[htbp]
\TABLE 
{{Computation results for instances of pure cross-category bundles with maximum bundle size three}\label{tab:ric_2}}
{\begin{tabular}{c l rrrrr rrrrr}
\toprule
\multicolumn{1}{c}{Instance} & \multirow{2}{*}{Model} 
& \multicolumn{5}{c}{Recursive McCormick} 
& \multicolumn{5}{c}{RMC \& Running-Intersection} \\
\cmidrule(r){1-1} \cmidrule(lr){3-7} \cmidrule(l){8-12}
$N$ & & \#sol & Time(s) & Gap(\%) & RGap(\%) & \#node & \#sol & Time(s) & Gap(\%) & RGap(\%) & \#node \\
\midrule
\multirow{3}{*}{100}   
& Pers & 9 & 348.2 (4.8) & 1.1 & 2.7 & 24684 & 10 & 61.2 (67.4) & 0.0 & 0.4 & 1428 \\
& Big-M & 0 & - (0.0) & 3.9 & 12.1 & 181968 & 0 & - (24.0) & 4.0 & 12.1 & 182549 \\
& CONIC & 0 & $\smallsetminus$ & $\smallsetminus$ & $\smallsetminus$ & $\smallsetminus$ & 0 & $\smallsetminus$ & $\smallsetminus$ & $\smallsetminus$ & $\smallsetminus$ \\
\midrule
\multirow{3}{*}{120} 
& Pers & 8 & 671.5 (10.6) & 1.3 & 2.5 & 20235 & 10 & 165.1 (197.1) & 0.0 & 0.4 & 2928 \\& Big-M & 0 & - (0.0) & 4.4 & 10.6 & 87927 & 0 & - (12.4) & 4.3 & 10.6 & 89345 \\
& CONIC & 0 & $\smallsetminus$ & $\smallsetminus$ & $\smallsetminus$ & $\smallsetminus$ & 0 & $\smallsetminus$ & $\smallsetminus$ & $\smallsetminus$ & $\smallsetminus$ \\
\midrule
\multirow{3}{*}{140}   
& Pers & 8 & 1945.6 (22.1) & 1.0 & 2.5 & 21319 & 10 & 89.3 (296.5) & 0.0 & 0.5 & 225 \\
& Big-M & 0 & - (0.0) & 5.0 & 10.0 & 45415 & 0 & - (14.4) & 5.0 & 10.0 & 45499 \\
& CONIC & 0 & $\smallsetminus$ & $\smallsetminus$ & $\smallsetminus$ & $\smallsetminus$ & 0 & $\smallsetminus$ & $\smallsetminus$ & $\smallsetminus$ & $\smallsetminus$ \\
\bottomrule
\end{tabular}}
{``-" in the Time(s) column indicates that the corresponding model fails to solve the MIP within 3600 seconds for all instances. ``$\smallsetminus$" for all metrics indicates that the corresponding model fails to solve the linear relaxation within 3600 seconds for all instances. The number in parentheses denotes the separation time. We fix $\delta=0.25$ and $\pi=1$.}
\end{table}
Additional sensitivity analyses in~\ref{ec:vary:graph} vary hypergraph density and the fraction of cross-category bundles. The results confirm that denser and more cross-category purchase structures tend to be computationally more difficult, while the perspective formulation remains effective across the tested settings.

\section{Extensions: Heterogeneous Customers and Robust Optimization}\label{sec:extend} 
In practice, different customer segments may exhibit distinct purchase patterns, and the estimated segment distribution may be subject to uncertainty due to limited data. In this section, we demonstrate that both the hypergraph representation and the perspective formulation can be extended to accommodate these practical considerations.
\subsection{Heterogeneous Customer Segments}\label{sec:extend:mix}
Customer heterogeneity is an important consideration in assortment planning. In a multi-purchase setting, customers may have different bundle consideration sets with distinct utilities. Our hypergraph representation can accommodate customer heterogeneity as follows. Suppose that the market consists of $K$ customer segments. Segment $k$ is characterized by a multi-purchase hypergraph $\G^k=(V,\E^k)$ and attraction values $v^k$ for each $k\in [K]$. Let $\lambda^k$ denote the proportion of segment $k$. The vector $\boldsymbol{\lambda} = (\lambda^1, \ldots, \lambda^K)$ therefore represents the customer segment distribution and must lie in the standard simplex $\Delta^K:= \bigl\{\boldsymbol{\lambda}\in [0,1]^K \bigm| \sum_{k\in [K]} \lambda^k=1 \bigr\}$. 

Given heterogeneous multi-purchase hypergraphs, we can define the assortment optimization problem under Mixed Logit-MP as follows. We introduce a variable $y_e^k$ such that $v_e^k y_e^k$ denotes the segment $k$'s choice probability of bundle $e$ for each segment $k\in [K]$ and bundle $e\in E^k$. Different segments are connected by the same assortment decision $\x\in\{0,1\}^{|V|}$. We solve the following binary fractional programming model to obtain an optimal assortment: 
\begin{equation*}
    \max_{\x \in\mathcal{X} \cap \{0,1\}^|V| }
\sum_{k\in [K]}\lambda^k
\sum_{e\in\E^k}r_e v_e^k y^k_e\quad \st y_e^k = \frac{\prod_{i\in e}x_i}{1+\sum_{c\in \E^k} v_c^k \prod_{i\in c}x_i} \for e\in \E^k, k\in [K]. \tag{\textsc{AOPT-mix}} \label{eq:ao-mix}
\end{equation*}
Customer heterogeneity introduces an additional layer of optimization complexity --- the retailer must select a single assortment that simultaneously affects the demand of all segments. 

Our framework naturally extends to this heterogeneous setting. 
The segment-level choice probabilities can be represented using the same perspective formulation developed in Section~\ref{sec:strong:model}, with segment-specific choice probability variables and shared assortment decisions. Specifically, for a given hypergraph $\G$ and a vector of attraction values $\v$, let $\mathcal{S}(\G,\v)$ denote the perspective formulation defined as in~\eqref{eq:pers:oracle}. Using this, we obtain the following MIP formulation for~\eqref{eq:ao-mix}: 
\begin{equation}\label{eq:mix:model}
        \max\limits_{\boldsymbol{\rho},\x,\boldsymbol{\y}} \biggl\{  \sum_{k\in [K]} \lambda^k \sum_{e\in \E^k} r_e   v_e^k  y_e^k \biggm| \x\in \mathcal{X}, (\rho^k, \y^k, \x) \in \mathcal{S}(\G^k,\v^k) \, \for k \in [K]   \biggr\}. \tag{\textsc{Pers-mix}}  
\end{equation}  

\subsection{Robust Assortment Optimization}\label{sec:extend:robust}
Customer segment proportions are estimated from historical transaction data and may contain estimation errors. A robust assortment model is therefore important when the customer distribution is uncertain. Following~\cite{bertsimas2017robust}, we consider a polyhedral uncertainty set $\mathcal{L}$ for the type distribution, that is, $\mathcal{L}:=\{\boldsymbol{\lambda}  \mid \BFB \boldsymbol{\lambda} \geq \BFd \}$,
where $\BFB \in \bbmr^{m\times K}$, $\BFd\in \bbmr^{m}$, and $\mathcal{L}\subseteq \Delta^K$. The retailer therefore solves the following robust counterpart of~\eqref{eq:ao-mix} 
\begin{equation*}
    \max_{\x\in \mathcal{X}  \cap \{0,1\}^|V| }
\min_{\lambda\in \mathcal{L}}
\sum_{k\in [K]}\lambda^k
\sum_{e\in\E^k}r_e v_e^k y^k_e \quad \st y_e^k = \frac{\prod_{i\in e}x_i}{1+\sum_{c\in \E^k}v_c^k \prod_{i\in c}x_i} \for e\in \E^k, k\in [K]. \tag\textsc{AOPT-robust} \label{eq:ao-robust}
\end{equation*} 

Our formulation framework also extends to this robust setting. For any fixed assortment and the associated segment-level choice probabilities, the inner minimization over the customer-segment distribution is a linear program because the uncertainty set $\mathcal{L}$ is polyhedral and the expected revenue is linear in $\boldsymbol{\lambda}$. We therefore dualize this inner problem: let $\boldsymbol{\eta}\in \bbmr_+^m$ be the dual variable and $B_k$ denote the $k^{\text{th}}$ column of $\boldsymbol{\BFB}$. By strong duality, the worst-case expected revenue can be represented through the corresponding dual variables and linear constraints linking them to the segment-level revenues. Combining these dual constraints with the perspective formulation for each segment yields the following formulation.
\begin{equation}\label{eq:robust_mix:model}
        \max \limits_{\x,\boldsymbol{\rho},\y, \boldsymbol{\eta}} \biggl\{  \BFd^{\top} \boldsymbol{\eta} \biggm|\begin{array}{lc}
             \boldsymbol{\eta}^{\top} B_k  =  \sum_{e\in \E^k} r_e  v_e^k  y_e^k,(\rho^k, \y^k, \x)\in \mathcal{S}(\G^k,\v^k)& \for k \in [K] \\
              \boldsymbol{\eta}\in \bbmr_+^m, \x\in \mathcal{X} & 
        \end{array} \biggr\}.  \tag{\textsc{Pers-robust}} 
\end{equation}

The following corollary establishes the exactness of the proposed formulations for both mixed and robust assortment optimization problems.
\begin{corollary}\label{coro:mix}
Formulations~\eqref{eq:mix:model} and~\eqref{eq:robust_mix:model} solve problems~\eqref{eq:ao-mix} and~\eqref{eq:ao-robust}, respectively.
\end{corollary}

We evaluate the computational performance of the robust heterogeneous extension in~\ref{ec:extend}, which represents the most general setting. The results show that increasing customer heterogeneity raises computational difficulty, while the proposed formulation solves most tested instances within the time limit, and running-intersection inequalities further reduce root gaps and branch-and-bound effort in the more challenging settings.

\section{Conclusion}\label{sec:conclude}
This paper develops an exact optimization framework for assortment planning under general Logit-MP models. The hypergraph representation reveals that the key difficulty arises from how assortment decisions activate overlapping purchase bundles. Our perspective formulation transfers strong descriptions of this bundle-activation structure to the choice-probability space, showing that the tractability of multi-purchase assortment planning depends on the topology of purchase interactions. The empirical study further shows that relatively sparse hypergraphs capture most multi-purchase behavior information and improve revenue, while the computational experiments demonstrate that the resulting formulations remain effective at large scale.

Several directions for future research remain. 
First, developing scalable estimation procedures for multi-purchase hypergraphs with statistical guarantees is a promising direction. Second, the proposed framework can be extended to other operational decisions involving bundle interactions, such as bundle design and joint pricing under multi-purchase behavior.


{\footnotesize
\setlength{\bibsep}{0pt}
\bibliographystyle{informs2014_r2}  
\bibliography{paper} 
}

%
\clearpage
%
%

\clearpage

\ECSwitch



\ECHead{Additional Results, Proofs, and Implementation Details}

\setcounter{section}{8} 
\renewcommand{\thesection}{EC.\the\numexpr\value{section}-8\relax} 


\setcounter{table}{8} 
\setcounter{figure}{8} 
\renewcommand{\thetable}{EC.\the\numexpr\value{table}-8\relax} 
\renewcommand{\thefigure}{EC.\the\numexpr\value{figure}-8\relax} 

\small 

\section{Proofs}
\subsection{Hardness Proof}
\begin{repeattheorem}
     The assortment optimization problem~\eqref{eq:assortment-single} is NP-hard to approximate within a factor of $O(\vert V \vert^{1-\epsilon})$ for any fixed $\epsilon \in (0,1)$, even when there is no constraint.
\end{repeattheorem}

Our hardness result is established via a reduction from the Maximum Independent Set (MIS) problem. Given an undirected simple graph $\G(V,\E)$, the MIS problem seeks a maximum-cardinality subset of vertices such that no two nodes are adjacent. It is a well-known NP-hard problem whose approximation is difficult. Specifically, finding an approximation solution $S\subseteq V$ such that $\frac{|S^*|}{|S|} \le O(|V|^{1-\epsilon})$ for any fixed $\epsilon \in (0,1)$ is NP-hard unless P = NP, where $S^*$ is the optimal maximum independent set. This result is formally stated in the following lemma. 
\begin{lemma}\label{lemma:inapprox}
    \citep{haastad1999clique} For any $\epsilon \in (0,1)$, it is NP-hard to approximate the Maximum Independent Set problem to within $N^{1-\epsilon}$, where $N$ is the number of nodes.
\end{lemma}

\noindent \textbf{Proof.} Given an MIS instance $G=(V, E_{MIS})$ with $|V|=N$, we construct an assortment optimization instance under the Logit-MP model as follows. We define the consideration set as $\E = V \cup E_{MIS}$. For each singleton bundle $i\in V$, we set its attraction value $v_i = 1/N^2$ and revenue $r_i = N$. For each edge bundle $(i,j) \in E_{MIS}$, we set its attraction value $v_{ij} = N^3$ and revenue $r_{ij} = 0$. This construction is clearly polynomial in $N$. 

For any offered assortment $S\subseteq V$, let $\beta(S)$ denote the number of edges induced by $S$, i.e.,
\[ \beta(S) = |\{(i,j) \in E_{MIS} \mid i,j\in S\}|. \]
The objective function of the constructed assortment optimization problem (AOP) equals:
\[ AOP(S) = \frac{N \cdot |S| \cdot (1/N^2)}{1 + |S|\cdot (1/N^2) + N^3 \cdot \beta(S)} = \frac{|S|/N}{1 + |S|/N^2 + N^3\beta(S)}, \]
where the zero-revenue edge bundles contribute exclusively to the denominator, acting as severe penalties for selecting adjacent nodes.

Next, we establish the bounds for the objective $AOP(S)$. Suppose $S\subseteq V$ is a feasible assortment. If $S$ is an independent set, then by definition $\beta(S) = 0$, which simplifies the objective to:
\begin{equation}\label{eq:apx:obj}
    AOP(S) = \frac{|S|/N}{1 + |S|/N^2}.
\end{equation} 
Since $|S| \le N$, we have $0\le |S|/N^2 \le 1/N$. This allows us to construct tight lower and upper bounds for any independent set $S$:
\begin{equation}\label{eq:apx:1}
    AOP(S) = \frac{|S|/N}{1 + |S|/N^2} \ge \frac{|S|/N}{1 + 1/N} = \frac{|S|}{N+1},
\end{equation}
and 
\begin{equation}\label{eq:apx:3}
    AOP(S) = \frac{|S|/N}{1 + |S|/N^2}  \le \frac{|S|}{N}.
\end{equation}
Conversely, if $S$ is \textit{not} an independent set, then $\beta(S) \ge 1$. Consequently, we can bound the objective from above:
\begin{equation}\label{eq:apx:2}
    AOP(S) \le \frac{|S|/N}{N^3\beta(S)} \le \frac{|S|/N}{N^3} \le \frac{1}{N^3},
\end{equation}
where the inequalities follow from the facts that $1+|S|/N^2 \ge 0$, $\beta(S) \ge 1$, and $|S|/N \le 1$.

Based on these bounds, we claim that for any $N \ge 2$, the optimal assortment $S^*_{AOP}$ must be an independent set. For the sake of contradiction, suppose $S^*_{AOP}$ is not an independent set. By \eqref{eq:apx:2}, we have $AOP(S^*_{AOP}) \le 1/N^3$. However, consider a trivial assortment $\{i\}$ consisting of any single product, which is inherently an independent set of size $1$. By \eqref{eq:apx:1}, its expected revenue is $AOP(\{i\}) \ge 1/(N+1)$. Since $1/(N+1) > 1/N^3$ strictly holds for all $N\ge 2$, we obtain $AOP(\{i\}) > AOP(S^*_{AOP})$, which contradicts the optimality of $S^*_{AOP}$. Thus, $S^*_{AOP}$ is guaranteed to be an independent set. 

Furthermore, observe that the function $f(x) = \frac{x/N}{1 + x/N^2}$ is strictly increasing for $x \in[0, N]$. Therefore, restricted to independent sets, $AOP(S)$ is strictly monotonically increasing in the cardinality $|S|$, see~\eqref{eq:apx:obj}. This structural property ensures that the optimal assortment exactly coincides with the maximum independent set, i.e., $S^*_{AOP} = S^*_{MIS}$.

Now, we establish the inapproximability result. We prove that any polynomial-time approximation algorithm for our assortment problem can be converted into an approximation algorithm for MIS, eventually violating Lemma \ref{lemma:inapprox}. Suppose, for contradiction, there exists a polynomial-time algorithm $\mathcal{A}$ that returns a feasible assortment $\hat{S}$ with an approximation ratio $\gamma \ge 2N^{-(1-\epsilon)}$, such that:
\begin{equation}\label{eq:apx:4}
    AOP(\hat{S}) \ge \gamma \cdot AOP(S^*_{AOP}).
\end{equation}
First, we must show that for sufficiently large $N$, $\hat{S}$ is forced to be an independent set. Since $S^*_{AOP}$ is optimal, $|S^*_{AOP}| \ge 1$, which implies $AOP(S^*_{AOP}) \ge 1/(N+1)$. Thus, the algorithm guarantees $AOP(\hat{S}) \ge \frac{\gamma}{N+1} = \frac{2}{N^{1-\epsilon}(N+1)}$. If $\hat{S}$ was not an independent set, we would have $AOP(\hat{S}) \le 1/N^3$. However, for any fixed $\epsilon > 0$ and sufficiently large $N$, $\frac{2}{N^{1-\epsilon}(N+1)} > \frac{1}{N^3}$, leading to a contradiction. Hence, $\hat{S}$ must be an independent set.

Because $\hat{S}$ is an independent set, we can safely apply the upper bound from \eqref{eq:apx:3}, yielding $|\hat{S}| \ge N \cdot AOP(\hat{S})$. We then derive:
\[
    |\hat{S}| \ge N \cdot AOP(\hat{S}) \ge N \cdot \gamma \cdot AOP(S^*_{AOP}) = N \cdot \gamma \cdot AOP(S^*_{MIS}).
\]
Using the lower bound from \eqref{eq:apx:1} for $AOP(S^*_{MIS})$, we obtain:
\[
    |\hat{S}| \ge N \cdot \gamma \cdot \frac{|S^*_{MIS}|}{N+1} \ge \frac{\gamma}{2} \cdot |S^*_{MIS}|,
\]
where the last step uses the fact that $\frac{N}{N+1} \ge \frac{1}{2}$ for all $N \ge 1$. Substituting $\gamma \ge 2N^{-(1-\epsilon)}$, we conclude that $|\hat{S}| \ge  |S^*_{MIS}|/N^{(1-\epsilon)}$. This implies we have discovered a polynomial-time algorithm that approximates the MIS problem within a factor of $O(N^{1-\epsilon})$. This directly contradicts Lemma \ref{lemma:inapprox}. Therefore, it is NP-hard to approximate the assortment optimization problem under the general Logit-MP model within a factor of $O(|V|^{1-\epsilon})$. \hfill \Halmos

\subsection{Choice Probability Set}
For the remaining formulation proofs, it is convenient to work with the set of choice probabilities induced by feasible assortments. Define 
\begin{equation}\label{eq:cg:set}
    \C_{\G}  =\{(\rho, \y)\mid \eqref{eq:ye}\,\eqref{eq:rho} \for \x \in \{0,1\}\}.
\end{equation}
Since the assortment objective is linear in $\y$, the Logit-MP assortment optimization problem reduces to 
\[\max \sum_{e\in \E}r_e v_e y_e \st (\rho, \y)\in \C_{\G}.\]
Therefore, exactness and relaxation strength can be established by comparing the projections of the corresponding formulation regions onto the $(\rho,\y)$-space.

\subsection{Proof of Theorem~\ref{theorem:main}} 
\begin{repeattheorem} 
For any hypergraph $\G$ and assortment feasible set $\mathcal{X}$, the formulation~\eqref{model:pers} exactly solves~\eqref{eq:assortment-single} if the bundle-activation relaxation $\R$ is exact. If the assortment problem is unconstrained and $\R$ is exact and integral, the LP relaxation of~\eqref{model:pers} exactly solves~\eqref{eq:assortment-single}.
\end{repeattheorem} 

\noindent \textbf{Proof.}  
Let $\mathcal{F}_{\text{pers}}$ be the feasible region of the linear relaxation of~\eqref{model:pers}. Our proof reduces to establish the relationship between $\C_{\G}$ and the projection $\proj_{(\rho,\y)}(\mathcal{F}_{\text{pers}})$.

\vspace{0.5em}
\noindent\textbf{Case 1: Exactness.} We first show that $\C_{\G} = \proj_{(\rho,\y)}(\mathcal{F}_{\text{pers}} \cap \{\x: \x\in \{0,1\}^{|V|}\})$.

\textbf{(Part 1: $\C_{\G} \subseteq  \proj_{(\rho,\y)}(\mathcal{F}_{\text{pers}} \cap \{\x: \x\in \{0,1\}^{|V|}\})$)}
Let $(\rho,\y)\in \C_{\G}$. By the definition of choice probabilities, there exists a binary assortment $\x \in \{0,1\}^{|V|}$ such that 
\[
\rho+\sum_{e\in\E}v_e y_e
= \frac{1}{1+\sum_{e\in\E}v_e\prod_{j\in e}x_j}
  + \sum_{e\in\E}v_e\frac{\prod_{j\in e}x_j}{1+\sum_{e\in\E}v_e\prod_{j\in e}x_j}
=1,
\]
which precisely satisfies the probability constraint~\eqref{eq:pers:oracle-2}.
Since $v_e = \exp(u_e) > 0$ for all $e\in\E$, it strictly follows that $\rho\in(0,1]$. Consequently, we have
$y_e / \rho  = \prod_{j\in e}x_j$ for all $e\in\E$. Because the relaxation $\R$ is exact for $\mathcal{M}_{\G}$, the vector $\z$ with $z_e = y_e/\rho$ satisfies constraint~\eqref{eq:pers:oracle-1}. 

Next, we verify constraints~\eqref{eq:pers:oracle-4} and~\eqref{eq:pers:oracle-5}. For any $i \in V$, let $L_i(\y)$ and $U_i(\y)$ denote the right-hand sides of~\eqref{eq:pers:oracle-4} and~\eqref{eq:pers:oracle-5}, respectively. Noting that $y_i = \rho x_i$ and $y_e = \rho \prod_{j\in e}x_j$, we evaluate these bounds based on the binary value of $x_i$:
\begin{itemize}
    \item If $x_i = 0$: Then $y_i = 0$, and for any bundle $e$ containing $i$, $y_e = 0$. For any bundle $e$ not containing $i$, $0\le y_e\le \rho$.  Thus, $\max\{0, y_i + y_e - \rho\} = 0$ and $\min\{y_i, y_e\} = 0$. Both $L_i(\y)$ and $U_i(\y)$ equal $0$, matching $x_i = 0$.
    \item If $x_i = 1$: Then $y_i = \rho$. For any bundle $e$ containing $i$, we have $\max\{0, y_i + y_e - \rho\} = \max\{0, y_e\} = y_e$, and $\min\{y_i, y_e\} = \min\{\rho, y_e\} = y_e$ (since $y_e \le \rho$). In this case, both $L_i(\y)$ and $U_i(\y)$ reduce to $\rho + \sum_{e \in \E \mid i \in e} v_e y_e + \sum_{e \in \E \mid i \notin e} v_e y_e = \rho + \sum_{e \in \E} v_e y_e$, which equals $1$, perfectly matching $x_i = 1$.
\end{itemize}
Since $L_i(\y) = U_i(\y) = x_i$ holds for both binary cases, constraints~\eqref{eq:pers:oracle-4} and~\eqref{eq:pers:oracle-5} are satisfied. Thus, $(\rho,\y,\x) \in \mathcal{F}_{\text{pers}}$.

\textbf{(Part 2: $\proj_{(\rho,\y)}(\mathcal{F}_{\text{pers}} \cap \{\x: \x\in \{0,1\}^{|V|}\}) \subseteq \C_{\G}$).} 
Conversely, suppose $(\rho,\y,\x) \in \mathcal{F}_{\text{pers}}$ with $\x \in \{0,1\}^{|V|}$. Define $\z \in \R$ and hence $y_e = \rho z_e \in [0,1]$. Since $z_e\in [0,1]$ and $\rho \ge 0$, we have $0 \le y_e \le \rho$ for all $e \in \E$. 
Because $\v > \boldsymbol{0}$ and the maximization terms are non-negative, constraint~\eqref{eq:pers:oracle-4} implies:
\begin{equation}\label{eq:mc2}
    x_i \ge y_i + \sum_{e\in \E \mid i\in e} v_e y_e  + \sum_{e\in \E \mid i\notin e} v_e \max\{0,y_i + y_e - \rho \} \ge y_i.
\end{equation}
Simultaneously, using the identity $\rho + \sum_{e\in\E} v_e y_e = 1$, constraint~\eqref{eq:pers:oracle-5} implies:
\begin{equation}\label{eq:mc3}
    x_i \le y_i + \sum_{e\in \E \mid i\in e} v_e y_e  + \sum_{e\in \E \mid i\notin e} v_e y_e = y_i + (1 - \rho).
\end{equation}
Coupling $y_i \le x_i \le y_i + 1 - \rho$ with the binary restriction $x_i \in \{0,1\}$, we can get $y_i = \rho x_i$ for all $i \in V$. Since $\R$ is exact, $\z \in \R$ implies $z_e = \prod_{j \in e} x_j$, giving $y_e = \rho \prod_{j \in e} x_j$. Substituting this into~\eqref{eq:pers:oracle-2} perfectly reconstructs the definition of $\C_{\G}$.

\vspace{0.5em}
\noindent\textbf{Case 2: LP exactness.} We now show that the $\mathcal{F}_{\text{pers}}$ projects exactly to the convex hull of $\C_{\G}$.

When $\x$ is relaxed to $x_i \in [0,1]$, we can apply Fourier-Motzkin elimination to project out $\x$. For each $x_i$, the system imposes two lower bounds ($0$ and $L_i(\y)$) and two upper bounds ($1$ and $U_i(\y)$). The projection is valid if and only if all lower bounds are less than or equal to all upper bounds. We verify this systematically:
First, $0 \le 1$ holds trivially. Second, since $y_e \ge 0$ and $v_e > 0$, $0 \le U_i(\y)$ is evident. Third, because $y_i, y_e \in [0,\rho]$, we have $\max\{0, y_i + y_e - \rho\} \le y_e$. Thus, $L_i(\y) \le y_i + \sum_{e\in \E} v_e y_e \le \rho + \sum_{e\in \E} v_e y_e = 1$. Finally, to see $L_i(\y) \le U_i(\y)$, we only need to verify that $\max\{0, y_i + y_e - \rho\} \le \min\{y_i, y_e\}$. This inequality holds because $y_i, y_e \le \rho$ implies $y_i + y_e - \rho \le y_i$ and $y_i + y_e - \rho \le y_e$. 
Therefore, projecting out $\x$ imposes \textit{no new constraints} on $(\rho,\y)$.  Define $\z \in \R$ and the projection is simply:
\[
\proj_{(\rho,\y)}(\mathcal{F}_{\text{pers}}) = \bigl\{(\rho,\y) \mid z_e = y_e/\rho, \z \in \R, \rho+\sum_{e\in\E}v_e y_e=1 \bigr\}.
\] 
By Theorem~2 of~\cite{he2024convexification}, we know
\[
\conv(\C_{\G}) = \bigl\{(\rho,\y)\mid (\rho,\y)\in\rho\cdot\conv(F),\ \rho+\sum_{e\in\E}v_e y_e=1,\ \rho\ge0\bigr\},
\]
where $F=\{(1,\z)\mid z_e=\prod_{j\in e}x_j,\ \x\in\{0,1\}^{|V|} , e\in \E\}$.
Since $\R$ is an exact and integral relaxation for $\mathcal{M}_{\G}$, we have $\conv(F)=\{(1,\z)\mid  \z \in \R\}$, which coincides with constraint~\eqref{eq:pers:oracle-1}.
\hfill\Halmos


\subsection{Proof of Propositions in Section~\ref{sec:strong:model}}
Before proving propositions about different bundle-activation relaxations in Section~\ref{sec:model}, we first introduce the \textit{standard linear relaxation} of the choice availability set $\M_{\G}$, defined as follows.
\begin{align*}
    &0\le z_e \le z_i \quad &\for i \in e, e\in \E\\
    & z_e \ge \sum_{i\in e} z_i - |e|+1 &\for   e\in \E.
\end{align*}
The standard linear relaxation yields an exact formulation for multilinear sets. As shown in the following lemma, it is inherently implied by the recursive McCormick inequalities.
\begin{lemma}\label{lemma:standard}
    The standard linear relaxation is implied by recursive McCormick inequalities. 
\end{lemma}

\noindent \textbf{Proof.}
The upper bound $0 \le z_e \le z_i$ for any $i\in e$ is directly imposed by the recursive McCormick inequalities. For the lower bound, we proceed by a straightforward induction on the bundle size $|e|$. Recursively applying the McCormick lower bound inequality $z_e \ge z_j + z_{e\setminus\{j\}} - 1$ over all elements $j \in e$ systematically accumulates the terms, yielding exactly $z_e \ge \sum_{i\in e}z_i - |e| + 1$. \hfill\Halmos

\vspace{1em}

\begin{repeatproposition}
  If the bundle-activation relaxation $\R$ is given by \eqref{eq:RMC} then the formulation~\eqref{model:pers} exactly solves~\eqref{eq:assortment-single}.
\end{repeatproposition}
\textbf{Proof.} When the decision variables are restricted to binary values $\x \in \{0,1\}^{|V|}$, the recursive McCormick inequalities exactly enforce the multilinear relationship $z_e = \prod_{i\in e} x_i$. Therefore, because the relaxation $\R$ perfectly captures $\M_{\G}$ at all integer points, $\R$ is an exact MIP formulation of $\M_{\G}$. The exactness of the perspective formulation then follows directly from Theorem~\ref{theorem:main}. \hfill \Halmos

\begin{repeatproposition}  
 If the bundle-activation relaxation $\R$ is given by \eqref{eq:RMC} and \eqref{eq:odd}, then the formulation~\eqref{model:pers} exactly solves~\eqref{eq:assortment-single} under any multi-purchase hypergraph $\G$. If $\G$ is a series-parallel graph, then its LP relaxation exactly solves unconstrained~\eqref{eq:assortment-single}.
\end{repeatproposition}
\textbf{Proof.} By construction, inequalities~\eqref{eq:odd} are valid for the multilinear
set $\M_{\G}$~\citep{padberg1989boolean}. Hence, by Theorem~\ref{theorem:main},~\eqref{model:pers} is an MIP formulation of~\eqref{eq:assortment-single}. 

Theorem~10 in \cite{padberg1989boolean} shows that if $\G$ is a series-parallel graph, the convex hull of $\M_{\G}$ is characterized by two types of inequalities. The first class comprises the trivial inequalities (i.e., the standard linear relaxation), which are implied by \eqref{eq:RMC} via Lemma~\ref{lemma:standard}. The second class consists of the nontrivial odd-cycle inequalities~\eqref{eq:odd}. The relaxation constructed by \eqref{eq:RMC} and \eqref{eq:odd} describes the convex hull of $\M_{\G}$, and thus is integral when $\G$ is a series-parallel graph. By Theorem~\ref{theorem:main}, the LP relaxation of~\eqref{model:pers} exactly solves the assortment optimization problem. 
\hfill \Halmos

\begin{repeatproposition} 
      If the bundle-activation relaxation $\R$ is given by \eqref{eq:RMC} and \eqref{eq:running}, the formulation~\eqref{model:pers} exactly solves~\eqref{eq:assortment-single} under any multi-purchase hypergraph $\G$. If  $\G$ is a kite-free $\beta$-acyclic hypergraph, then its LP relaxation exactly solves unconstrained~\eqref{eq:assortment-single}.
\end{repeatproposition}  
\textbf{Proof.} Proposition~10 in \cite{del2021running} establishes that the running-intersection inequalities~\eqref{eq:running} are valid for the multilinear set $\M_{\G}$. Thus, by Theorem~\ref{theorem:main},~\eqref{eq:pers:oracle} is an MIP formulation of~\eqref{eq:assortment-single}. 
 
Theorem~3 in \cite{del2021running} shows that if $\G$ is a kite-free $\beta$-acyclic hypergraph, the convex hull of $\M_{\G}$ is characterized by two types of inequalities: the trivial inequalities (implied by \eqref{eq:RMC} via Lemma~\ref{lemma:standard}) and the nontrivial running-intersection inequalities~\eqref{eq:running}. Consequently, \eqref{eq:RMC} and \eqref{eq:running} jointly constitute an integral formulation of $\M_{\G}$. The LP exactness of the perspective formulation follows directly from Theorem~\ref{theorem:main}.\hfill \Halmos

\subsection{Proof of Proposition~\ref{prop:bigm:exact}}
\begin{repeatproposition}
If the bundle-activation relaxation $\R$ is exact, then the formulation~\eqref{model:bigm} is an MIP formulation for the assortment optimization problem~\eqref{eq:assortment-single}.
\end{repeatproposition}
\textbf{Proof.} Because the formulation~\eqref{model:bigm} and problem~\eqref{eq:assortment-single} have the same objective function, it reduces to prove that~\eqref{eq:bigm} is an MIP formulation of~$\C_{\G}$. We first show $\proj_{(\rho,\y)}\eqref{eq:bigm} \subseteq \C_{\G}$. Let $(\rho, \y, \x, \z)$ be a feasible solution in $\eqref{eq:bigm}$. The binary restriction on $\x$ combined with the McCormick inequalities~\eqref{eq:bigm-1} enforces $y_e = \rho z_e$ for each $e\in \E$. Because $\z \in \R$ and $\x$ is binary, the exactness of the relaxation $\R$ guarantees $z_e = \prod_{i\in e} x_i$. Substituting these relations into the probability constraint $\rho + \sum_{e\in \E} v_e y_e = 1$ of \eqref{eq:bigm}, we obtain:
\[ \rho + \sum_{e\in \E} v_e \bigl(\rho \cdot \prod_{i\in e} x_i \bigr)= 1.\]
Since $\v$ and $\x$ are non-negative, the term $1+\sum_{e\in \E} v_e \prod_{i\in e} x_i \ge 1 > 0$. This ensures the denominator is strictly positive, allowing us to safely factor out $\rho$ and rearrange the equation as:
\[\rho = \frac{1}{1+\sum_{e\in \E} v_e \prod_{i\in e} x_i} , y_e =\frac{\Pi_{i\in e} x_i}{1+\sum_{e\in \E} v_e \Pi_{i\in e} x_i} \for e \in \E, \]
which is exactly the definition of $\C_{\G}$.

Conversely, we prove $\C_{\G} \subseteq \proj_{(\rho,\y)}\eqref{eq:bigm}$. Suppose $(\rho, \y)$ is a feasible solution of~\eqref{eq:cg:set}. By the definition of $\C_{\G}$, there exists a binary assortment vector $\x \in \{0,1\}^{|V|}$ generating these probabilities. Consequently, we have:
\[ \rho + \sum_{e\in \E} v_e y_e =  \frac{1}{1+\sum_{e\in \E} v_e \Pi_{i\in e} x_i } + \sum_{e\in \E} v_e\frac{\Pi_{i\in e} x_i}{1+\sum_{e\in \E} v_e \Pi_{i\in e} x_i } = 1.\]  
Note that all attraction values are positive, i.e., $v_e > 0 \for e\in \E$, and hence $\rho\in (0,1]$. Let $z_e = \prod_{i\in e} x_i$ for each $e\in \E$. Because $\x$ is binary and the relaxation $\R$ is exact, we obtain $\z \in \R$, satisfying constraint~\eqref{eq:bigm:0}. By the definition of $\y$, we have $y_e = \rho z_e$ for all $e\in \E$. Furthermore, $\x$ explicitly enforces the integrality condition $\z \in \{0,1\}^{|\E|}$. Taken together, these two conditions ensure that the Big-M relaxation~\eqref{eq:bigm-1} is satisfied. Therefore, $(\rho, \y, \x,\z) \in \eqref{eq:bigm}$, concluding the proof. \hfill \Halmos

\subsection{Proof of Theorem~\ref{thm:compare}} 
\begin{repeattheorem} 
For any relaxation $\R$, $Z_{\texttt{pers}} \le Z_{\texttt{Big-M}}$.  
\end{repeattheorem}
  \textbf{Proof.} 
Because two LP relaxations have the same objective function, it suffices to show that a feasible solution of LP relaxation of~\eqref{model:pers} is also feasible for LP relaxation of~\eqref{model:bigm}.
Suppose $(\rho,\x,\y)$ is a feasible solution of the LP relaxation of the perspective formulation~\eqref{model:pers}. Coupling the probability identity $\rho +\sum_{e \in \E} v_e y_e =1 $ with the conditions $0 \le y_e \le \rho$ for all $e\in \E$ and $\v > \boldsymbol{0}$, it strictly follows that $\rho \in (0,1]$. Because $\rho$ is strictly positive, we can well-define $z_e := y_e/\rho$ for each $e\in \E \setminus V$. 

Next, we demonstrate that $(\rho,\x,\y, \z)$ satisfies all constraints of the LP relaxation of~\eqref{model:bigm}. The inherent condition $0 \le y_e \le \rho$ immediately guarantees that $z_e \in [0,1]$. Furthermore, since $(\rho,\x,\y)$ satisfies the scaled constraints of the perspective formulation, unscaling these constraints by the positive scalar $\rho$ ensures that $\z \in \R$, which satisfies constraint~\eqref{eq:bigm:0}.

We now verify that $(\rho,\x,\y, \z)$ satisfies the Big-M relaxation constraints for $y_e = \rho z_e$ by substituting the natural domain bounds $\rho_L=0$ and $\rho_U=1$:
\begin{itemize}
    \item Substituting $\rho_L=0$, the non-negativity $y_e \ge 0$ directly yields $y_e \ge \rho_L z_e$, while $y_e \le \rho$ equivalently yields $y_e \le \rho_L z_e + \rho - \rho_L$.
    \item Substituting $\rho_U=1$, since $\rho \le 1$ and $z_e \ge 0$, we safely obtain $y_e = \rho z_e \le z_e = \rho_U z_e$.
    \item Finally, because both $\rho \le 1$ and $z_e \le 1$, the product $(1-\rho)(1-z_e)$ must be non-negative. Expanding $(1-\rho)(1-z_e) \ge 0$ yields $\rho z_e \ge z_e + \rho - 1$, which exactly corresponds to $y_e \ge \rho_U z_e + \rho - \rho_U$.
\end{itemize}

Since all constraints of the linear relaxation of \eqref{eq:bigm} are satisfied, we conclude that $Z_{\texttt{pers}} \le Z_{\texttt{Big-M}}$.

\hfill\Halmos

\subsection{Proof of Corollary~\ref{coro:mix}} 
\begin{repeatcorollary}
Formulations \eqref{eq:mix:model} and \eqref{eq:robust_mix:model} solve problems \eqref{eq:ao-mix} and \eqref{eq:ao-robust}, respectively.
\end{repeatcorollary}

\noindent \textbf{Proof:}
Analogous to the single-class setting, we define variable $y^k_e$ such that $v^k_e y^k_e$ denotes the purchase choice probability and the no-purchase probability $\rho^k$ for each customer class $k$. The proof idea remains identical: the choice probabilities for each class $k$ must validly and independently reside within their respective probability space $\C_{\G^k}$. By replacing the choice probabilities of each class with $(\rho^k, \y^k, \x) \in \mathcal{S}(\G^k,\v^k)$, formulations~\eqref{eq:mix:model} and~\eqref{eq:robust_mix:model} directly emerge as exact reformulations of their original counterparts. \hfill \Halmos

\section{Additional Details for Numerical Experiments on Real Data}
\label{ec:est:support}

\subsection{Estimation Details for the Bakery Study}\label{ec:estimation:details}
\textbf{Data construction and cross-validation.} We convert the raw transactions into daily assortment--purchase observations. Following \cite{akchen2021assortment}, we use positive sales to construct the daily assortment proxy $S_t$: a product is treated as available on day $d$ if it is purchased at least once on that day. Each transaction $\tau$ is then represented by the pair $(S_t,e_\tau)$, where $e_\tau$ is the purchased bundle. We randomly partition the days into five folds. For each fold, the candidate hypergraph $\E_\theta$ is constructed using only the training days by retaining all singleton products and the most frequently observed multi-item bundles. When a bundle transaction is not explicitly represented in $\E_\theta$, we enumerate all admissible partitions of the purchased bundle into mutually disjoint bundles in $\E_{\theta}$ and assign equal weight to each partition. This splitting rule is also applied to the held-out fold.

\textbf{Estimation and evaluation.} Because product prices exhibit no variation during the study period, we omit price covariates and estimate the Full and Pairwise utility specifications by maximum likelihood. The dataset does not contain no-purchase transactions. We therefore use expectation-maximization (EM) procedure to handle this censored data issue.  Out-of-sample metrics are then calculated on each held-out fold. 

\subsection{Consideration Set Support Condition}\label{ec:support:recovery}
This section establishes a sufficient condition under which the observed bundle set recovers the latent consideration set as the sample size grows. The key idea is that if every true bundle is available at least once, then a sufficiently large number of transactions makes the probability of never observing that bundle vanish exponentially. We summarize the data as $(S_t,n_t)_{t\in [T]}$, where $S_t$ and $n_t$ are assortment and number of transactions on day $t\in [T]$, respectively. Let $\E^*$ be the true consideration set. We consider the following assumption. 

\begin{assumption}\label{ass:bundle:support}
For each bundle $e\in\E^*$, there exists at least one period $t(e)\in[T]$ such that $e\subseteq S_{t(e)}$.
\end{assumption}
The assumption rules out bundles that are never offered, leading to the following result.
\begin{proposition} \label{prop:support:recovery} 
Under Assumption~\ref{ass:bundle:support}, the probability that the observed bundle set $\E_{\mathrm{data}}$ misses any latent true bundle is bounded.  
\[
\Pr\left(
\E^*\nsubseteq \E_{\mathrm{data}}
\,\middle|\,
(S_t,n_t)_{t\in[T]}
\right)
\le
|\E^*|\exp(-\underline p n_{\mathrm{min}}),
\]
where $\underline{p}$ is the smallest purchase probability, i.e., $\underline p =
\min_{e\in\E^*}\frac{v_e}{1 + \sum_{c\in\E^*}v_c}$. $n_{\mathrm{min}}$ is the minimum sample size over $T$ periods, i.e., $n_{\mathrm{min}} = \min\{n_1,\dots, n_T\}.$ Consequently, this probability vanishes as the sample size grows:
\[
\Pr\left(
\E^*\nsubseteq \E_{\mathrm{data}}
\,\middle|\,
(S_t,n_t)_{t\in[T]}
\right)
\to 0
\qquad
\text{as }n_{\mathrm{min}}\to\infty.
\]
\end{proposition}

\noindent \textbf{Proof:}
Fix a true bundle $e\in\E^*$. 
By Assumption~\ref{ass:bundle:support}, there exists a period $t(e)$ such that $e\subseteq S_{t(e)}$. 
If $e$ is not observed in the data, then it must not be purchased in period $t(e)$. 
Therefore,
\[
\Pr\left(
e\notin\E_{\mathrm{data}}
\,\middle|\,
(S_t,n_t)_{t\in[T]}
\right) \le \bigl(1- \pr_{\E^*}(e|S_{t(e)})\bigr)^{n_{t(e)}}
\le
(1-\underline p)^{n_{t(e)}}.
\]
We obtain
\[
(1-\underline p)^{n_{t(e)}}
\le
\exp(-\underline p n_{t(e)})
\le
\exp(-\underline p n_{\min}).
\]
Applying the union bound over all $e\in\E^*$ gives
\[
\Pr\left(
\E^*\nsubseteq \E_{\mathrm{data}}
\,\middle|\,
(S_t,n_t)_{t\in[T]}
\right)
\le
\sum_{e\in\E^*}
\Pr\left(
e\notin\E_{\mathrm{data}}
\,\middle|\,
(S_t,n_t)_{t\in[T]}
\right)
\le
|\E^*|\exp(-\underline p n_{\min}).
\]
The convergence result follows from $n_{\min}\to\infty$. \Halmos

Proposition~\ref{prop:support:recovery} shows that, under bundle coverage assumption, the observed support $\E_{\mathrm{data}}$ contains all true bundles with probability approaching one. 
This result does not imply that every candidate hypergraph $\E_\theta$ contains all bundles in $\E^*$, because our estimation procedure deliberately removes low-frequency bundles to control model complexity.

\subsection{Semi-Synthetic Data Generation}
\label{ec:sushi:generation}
We construct semi-synthetic multi-purchase data by combining the Sushi preference data of \citet{kamishima2003nantonac} with the empirical bundle-size distribution from the bakery data. The Sushi dataset contains preference rankings from 5,000 respondents over ten sushi products. Each respondent provides a complete ranking of the ten products but does not rank the no-purchase option. Following the Top-3 specification of \citet{berbeglia2022comparative} and \citet{akchen2021assortment}, we truncate each respondent's ranking to her three most preferred products. Let $\sigma_j=(\sigma_{j1},\sigma_{j2},\sigma_{j3})$ denote the resulting ranking for respondent $j$. Thus, respondent $j$ considers only the products in $\sigma_j$ when making a purchase decision.

We generate transactions from these trimmed rankings as follows. Following \citet{berbeglia2022comparative} and \citet{akchen2021assortment}, for each transaction $t$, we draw an offered assortment $S_t$ uniformly at random from
$
\mathcal S_{\mathrm{sample}}
=
\left\{
S\subseteq V:3\le |S|\le6
\right\}.
$
We then independently draw a respondent $j_t$ uniformly from the 5,000 respondents and a purchase quantity $K_t$ from the empirical positive-purchase bundle-size distribution in the bakery data,
$$
K_t\sim \widehat F_K,
\qquad
\widehat F_K(k)
=
\Pr_{\mathrm{bakery}}(|e|=k\mid |e|>0).
$$
Let $A_t=S_t\cap\sigma_{j_t}$ denote the respondent's top-three products that are available in the offered assortment. The respondent chooses the $m_t=\min\{K_t,|A_t|\}$ highest-ranked products in $A_t$, so that
$
e_t
=
\operatorname{Top}_{m_t}(A_t;\sigma_{j_t}).
$
Thus, the Sushi rankings determine which products are preferred, while the bakery data determine how many products the customer intends to purchase; these two components are sampled independently. If $A_t=\varnothing$, the customer makes no purchase. To study different levels of purchase incidence, we additionally introduce a no-purchase customer type: with probability $q_0$, the transaction is assigned to the no-purchase customer, and otherwise the purchase is generated according to the rule above. We generate five independent datasets using this procedure for each configuration. Each dataset has 10,000 samples. The resulting choice mechanism is treated as the ground-truth model when evaluating the revenue of assortments obtained from the estimated choice models.

\section{Additional Numerical Results on Synthetic Data}\label{ec:numerical:synthetic}
\subsection{Numerical Study: Sensitivity Analysis}\label{ec:vary:graph}
We additionally examine how the multi-purchase hypergraph structures affect computational performance of our perspective formulation. We vary $\delta\in \{0.15, 0.2,0.25,0.3\}$ and the proportion of cross-category bundles $\pi\in \{0.2,0.4,0.6,0.8\}$ to control the complexity of the consideration set. We fix the number of products as $N=100$ and generate 10 instances for each configuration. Table~\ref{tab:sensetive} shows that denser hypergraphs and a higher proportion of cross-category bundles generally lead to larger computational challenges.
Nevertheless, the perspective formulation remains effective across all settings.

\begin{table}[htpb]
\TABLE
{{Performance of perspective formulation under different sparsity levels and cross-category ratios with $d=2$}\label{tab:sensetive}}
{\begin{tabular}{cc rrr cc rrr}
\toprule
\multicolumn{5}{c}{Low Density} &   \multicolumn{5}{c}{High Density} \\
\cmidrule(r){1-2} \cmidrule(lr){3-5} \cmidrule(r){6-7}  \cmidrule(l){8-10}
$\delta$ & $\pi$ & RGap(\%) & SolTime(s)  & SepTime(s) & $\delta$ & $\pi$ & RGap(\%) & SolTime(s) & SepTime(s) \\
\midrule
\multirow{4}{*}{0.15} & 0.2 &   1.01 & 2.23 & 1.60 & \multirow{4}{*}{0.25} & 0.2 & 1.48 & 22.74 & 6.47 \\
 & 0.4 & 1.70 & 11.18 & 3.51    & & 0.4  & 1.64 & 43.89 & 11.23\\
 & 0.6 & 1.53 & 37.79 & 6.88  && 0.6  & 1.16 & 77.77 & 30.26 \\
 & 0.8 & 1.38 & 50.77 & 16.39 && 0.8  & 0.76 & 35.67 & 42.75 \\
\midrule
\multirow{4}{*}{0.2} & 0.2 & 2.18 & 47.07 & 5.06  & \multirow{4}{*}{0.3} & 0.2 & 1.88 & 39.38 & 7.09 \\
 & 0.4  & 1.72 & 36.88 & 6.76 & & 0.4 & 1.45 & 72.40 & 19.97 \\
 & 0.6  & 1.59 & 72.59 & 19.26 & & 0.6 & 1.30 & 100.67 & 35.54\\
 & 0.8  & 0.99 & 77.82 & 25.72 & & 0.8 & 0.57 & 35.42 & 53.85 \\
\bottomrule
\end{tabular}}
{}
\end{table}

\subsection{Numerical Study: Robust Optimization}\label{ec:extend} 

We evaluate the computational performance of the robust heterogeneous extension, which represents the most general setting considered in this paper. We generate synthetic instances following the procedure described in Section~\ref{sec:app}. We vary the number of customer segments $K\in\{2,4\}$. For each segment, we generate a segment-specific multi-purchase hypergraph by sampling candidate bundles. Specifically, let $\E_d$ denote the set of candidate bundles with maximum cardinality $d$. 
For each segment $k$, a candidate bundle $e\in\E_d$ is included if $0.4 z_e^k+0.6z_e^0\leq 0.25$, where $z_e^k\sim\mathcal U(0,1)$ for $k\in\{0,1,\ldots,K\}$. 
This construction introduces similarity across customer segments through the common component $z_e^0$, while allowing segment-specific purchase patterns through $z_e^k$. The uncertainty set for the customer segment distribution is defined as
\[
\mathcal{L}=
\left\{
\boldsymbol{\lambda}\mid
\lambda^k\in[L^k,U^k],\ 
\sum_{k\in[K]}|\lambda^k-\hat{\lambda}^k|\leq\Gamma,\ 
\boldsymbol{\lambda}\in\Delta^K
\right\}.
\]
The nominal distribution $\hat{\boldsymbol{\lambda}}$ is generated by independently sampling each component from $\mathcal U(0,1)$ and normalizing the resulting vector to the simplex. 
We set $L^k=0.95\hat{\lambda}^k, 
U^k=1.05\hat{\lambda}^k$, and $
\Gamma=0.1$.
The assortment constraint is $\mathcal X=
\left\{
\x\in [0,1]^{|V|}
\mid
\sum_{i\in V}x_i\leq0.2|V|
\right\}$.

The computational results are reported in Table~\ref{tab:mix_1} and~\ref{tab:mix_2}, respectively. 
As expected, increasing the number of segments increases computational difficulty.
Nevertheless, the proposed formulation remains effective and solves most instances to optimality within the time limit. 
The results also confirm the importance of running-intersection inequalities, which substantially reduce the root gaps and branch-and-bound effort in more complex settings. 
 
\begin{table}[H]
\TABLE 
{{Computation results for pure intra-category bundles with $d=3$ and $\theta= 0.4$}\label{tab:mix_1}}
{\begin{tabular}{cc rrrrr rrrrr}
\toprule
\multicolumn{2}{c}{Instance}
& \multicolumn{5}{c}{Recursive McCormick} 
& \multicolumn{5}{c}{RMC\& Running-Intersection} \\
\cmidrule(r){1-2} \cmidrule(lr){3-7} \cmidrule(l){8-12}
$N$ & $K$ & \#sol & Time(s) & Gap(\%) & RGap(\%) & \#node & \#sol & Time(s) & Gap(\%) & RGap(\%) & \#node \\
\midrule
\multirow{2}{*}{70} & 2 & 10 & 21.2 (2.1) & 0.0 & 3.3 & 10401 & 10 & 17.5 (4.0) & 0.0 & 2.2 & 9568 \\
&4 & 10 & 172.0 (6.3) & 0.0 & 4.8 & 12602 & 10 & 138.4 (11.9) & 0.0 & 3.6 & 9667 \\
\midrule
\multirow{2}{*}{80} & 2 & 10 & 22.3 (3.6) & 0.0 & 2.9 & 4278 & 10 & 15.5 (6.1) & 0.0 & 2.0 & 2916 \\
&4 & 8 & 478.8 (11.6) & 2.0 & 5.7 & 40781 & 8 & 370.7 (23.0) & 1.6 & 4.0 & 34476 \\
\midrule
\multirow{2}{*}{90} & 2 & 10 & 43.4 (6.1) & 0.0 & 2.4 & 10078 & 10 & 28.8 (10.5) & 0.0 & 1.4 & 4106 \\
&4 & 6 & 845.8 (21.4) & 2.6 & 6.4 & 36458 & 6 & 559.5 (39.0) & 1.8 & 4.9 & 23497 \\
\bottomrule
\end{tabular}}
{}
\end{table}

\begin{table}[H]
\TABLE 
{{Computation results for pure cross-category bundles with $d=3$ and $\theta= 0.4$}\label{tab:mix_2}}
{\begin{tabular}{cc rrrrr rrrrr}
\toprule
\multicolumn{2}{c}{Instance}
& \multicolumn{5}{c}{Recursive McCormick} 
& \multicolumn{5}{c}{RMC\& Running-Intersection} \\
\cmidrule(r){1-2} \cmidrule(lr){3-7} \cmidrule(l){8-12}
$N$ & $K$ & \#sol & Time(s) & Gap(\%) & RGap(\%) & \#node & \#sol & Time(s) & Gap(\%) & RGap(\%) & \#node \\
\midrule

\multirow{2}{*}{40} 
& 2 & 10 & 8.6 (1.2) & 0.0 & 3.3  & 4397 & 10 & 6.5 (2.7) & 0.0 & 2.2 & 3385 \\
& 4 & 10 & 104.6 (3.9) & 0.0 & 5.0  & 12023 & 10 & 61.0 (7.8) & 0.0 & 3.6 & 8180 \\
\midrule
\multirow{2}{*}{50} 
& 2 & 10 & 36.4 (4.2) & 0.0 & 3.5 & 6779 & 10 & 18.6 (7.9) & 0.0 & 2.2 & 3510 \\
& 4 & 10 & 423.8 (13.1) & 0.0 & 5.9  & 15817 & 10 & 333.0 (26.0) & 0.0 & 4.5 & 12050 \\
\midrule
\multirow{2}{*}{60} 
& 2 & 10 & 222.0 (13.0) & 0.0 & 3.7 & 16710 & 10 & 120.8 (23.9) & 0.0 & 2.1 & 8818 \\
& 4 & 8 & 1903.4 (43.2) & 1.0 & 5.8 & 27345 & 10 & 1812.4 (76.3) & 0.0 & 4.2 & 27289 \\
\bottomrule
\end{tabular}}
{}
\end{table}
\section{Implementation Details}\label{ec:sec:algo}
Implementing the perspective formulation~\eqref{model:pers} presents two computational challenges. First,~\eqref{model:pers} contains $\mathcal{O}(|V|2^{|\E\setminus V|})$ inequalities for bounding $\x$. Second, the linear relaxation $\R$ may also include a large number of inequalities, particularly when the multi-purchase hypergraph is dense. To address this issue, we design a cutting-plane implementation. Our approach begins with a simplified base formulation containing fewer constraints. We then iteratively employ separation oracles for the two classes of inequalities, thereby tightening the formulation. Finally, we solve the tightened formulation with binary constraints to obtain an optimal integer solution. We detail the base formulation and separation oracles for constraints~\eqref{eq:pers:oracle-4} and \eqref{eq:pers:oracle-5}, odd-cycle inequalities, and running-intersection inequalities in Section~\ref{sec:base} and present the complete cutting-plane procedure in Section~\ref{sec:cut}. 

\subsection{Base Formulation and Separation Oracles}\label{sec:base}
We construct a base formulation with limited number of inequalities as follows. We first construct the base relaxation $\R$ by associating each bundle $e$ with a left-deep binary tree. We build binary trees via the following Algorithm~\ref{alg:rmc}.
\begin{algorithm}[htbp]
\small 
\caption{\small Construction of Binary Trees for Recursive McCormick Inequalities}
\label{alg:rmc}
\begin{algorithmic}[1]
\REQUIRE A multi-purchase hypergraph $\G(V,\E)$; an auxiliary-variable registry $\mathcal{W}$
\ENSURE  Binary trees for each hyperedge $e\in \E$ with $|e|\ge 2$ 
\STATE  Initialize each hyperedge $e = \{i_1, \dots, i_k\}$ with indices sorted
         in ascending order $i_1 <  \cdots < i_k$; initialize tree structure $\mathcal{T}$
\STATE Initialize an un-visited set $S = \{e: |e|\ge 2, e\in \E\}$ and a visited set $H\gets \emptyset$
\WHILE{ $S$ is non-empty}
 \STATE Find a root node $e = \argmin_{e\in S} |e|$ and $S \gets S\setminus \{e\}, H \gets H \cup \{e\}$   
\STATE Define $e^{\prime} = \{ i_1, i_2,\dots, i_{|e|-1}\}$ as the left child (the subtree associated with $e^{\prime}$) 
\STATE Define $j= i_{|e|}$ as the right child (the right node corresponding to $x_j$) 
\STATE $\mathcal{T}\gets \mathcal{T} \cup (e,e^{\prime},j)$
\IF{$|e^\prime|>1$ and $e^\prime \notin H$}
\STATE $S\gets S \cup \{e^{\prime}\}$ and $\mathcal{W}\gets \mathcal{W} \cup \{e^{\prime}\}$
\ENDIF  
\ENDWHILE
\RETURN $\mathcal{T}$ and $\mathcal{W}$ 
\end{algorithmic}
\end{algorithm}
The recursive McCormick inequalities are then imposed along the edges of this tree.
We refer to the relaxation with these recursive inequalities as $\R_{\text{RMC}}$.

Regarding constraints~\eqref{eq:pers:oracle-4} and \eqref{eq:pers:oracle-5}, we initially incorporate only the following extreme cases:
\begin{subequations}\label{eq:pers:x:simple}
    \begin{alignat}{3} 
   & x_i \geq  y_i + \sum_{e \in \E | i\in e } v_ey_e    && \quad \for i \in V \label{eq:base1} \\
   & x_i \geq  y_i + \sum_{e \in \E | i\in e } v_ey_e + \sum_{e \in E | i\notin e} v_e  ( y_i +y_e -\rho)  && \quad \for i \in V  \label{eq:base2}\\
    & x_i \le  y_i + \sum_{e \in \E | i\in e }v_e y_e + \sum_{e \in E | i\notin e} v_e  y_i    && \quad \for i \in V \label{eq:base3}\\ 
    & x_i \le  y_i + \sum_{e \in \E | i\in e } v_e y_e + \sum_{e \in E | i\notin e} v_e  y_e    && \quad \for i \in V. \label{eq:base4}
    \end{alignat}
\end{subequations}
Combining these elements yields the following base formulation: 
\begin{equation}\label{eq:base}
 \max \limits_{\x \in \mathcal{X}}\sum_{e\in \E} r_e v_e y_e \quad \st \rho + \sum_{e\in \E} v_e y_e = 1 , \rho \ge 0, \y \in \rho \cdot \R_{\text{RMC}}, \eqref{eq:pers:x:simple}, \x \in \{0,1\}^{|V|}. \tag{\texttt{Base}}
\end{equation}
The next proposition establishes the validity of the base formulation.  
\begin{proposition}\label{prop:base}
Formulation~\eqref{eq:base} is an MIP formulation for~\eqref{eq:assortment-single}.
\end{proposition}
\noindent \textbf{Proof.} The proof follows a simple sandwich argument. Let $R_{\mathrm{pers}}$ and $R_{\mathrm{base}}$ denote the LP feasible regions of~\eqref{model:pers} and~\eqref{eq:base}, respectively, and let $\bar R_{\mathrm{Big-M}}$ denote the projection of the LP feasible region of~\eqref{model:bigm} onto the $(\rho,\y,\x)$-space. We show that
\[R_{\mathrm{pers}}\subseteq R_{\mathrm{base}}\subseteq\bar R_{\mathrm{Big-M}}.\]
The first inclusion follows because~\eqref{eq:base} retains only a subset of the linking constraints in~\eqref{model:pers}. We establish the second inclusion by lifting any point in $R_{\mathrm{base}}$ to a feasible solution of the Big-M formulation. Since both~\eqref{model:pers} and~\eqref{model:bigm} are exact MIP formulations, the integer projection of~\eqref{eq:base} is sandwiched between them, proving its exactness.

Let $(\rho, \y, \x)$ be an arbitrary feasible point in the LP relaxation of~\eqref{eq:base}. By the definition of the perspective scaled set $\y \in \rho \cdot \R_{\text{RMC}}$, there exists a vector $\z \in \R_{\text{RMC}}$ such that $y_e = \rho z_e $ for all $e\in \E$. Because $\z \in \R_{\text{RMC}}$, the variables satisfy $z_e \in [0,1]$ for each $e\in \E$. Coupling this with the non-negativity constraint $\rho \ge 0$, we immediately obtain $0 \le y_e \le \rho$. Furthermore, since $v_e > 0$ and $y_e \ge 0$ for all $e \in \E$, the probability identity constraint $\rho + \sum_{e \in \E} v_e y_e = 1$ strictly ensures that $\rho \le 1$.

Now, we verify that the point $(\rho, \y, \x, \z)$ satisfies the LP relaxation of~\eqref{model:bigm}. Constraint~\eqref{eq:bigm:0} holds by construction. We verify the Big-M envelopes~\eqref{eq:bigm-1} as follows:
\begin{itemize}
    \item The bounds $0 \le y_e \le \rho$ inherently satisfy $y_e \ge \rho_L z_e$ and $y_e \le \rho_L z_e + \rho - \rho_L$.
    \item Since $\rho \le 1$ and $z_e \ge 0$, we have $y_e = \rho z_e \le 1 \cdot z_e = z_e$, perfectly matching $y_e \le \rho_U z_e$.
    \item Since both $\rho \le 1$ and $z_e \le 1$, their complements are non-negative, yielding $(1-\rho)(1-z_e) \ge 0$. Expanding this product gives $\rho z_e \ge z_e + \rho - 1$. Substituting $y_e = \rho z_e$ directly yields $y_e \ge z_e + \rho - 1$, which aligns exactly with $y_e \ge \rho_U z_e + \rho - \rho_U$.
\end{itemize}

Since all inequalities in~\eqref{model:bigm} are satisfied, we have $R_{\mathrm{base}}\subseteq\bar R_{\mathrm{Big-M}}$. \hfill \Halmos

\vspace{1 em}

While valid, formulation~\eqref{eq:base} can be further tightened by iteratively adding the remaining inequalities from the constraint pool of~\eqref{eq:pers:oracle-4} and \eqref{eq:pers:oracle-5} or stronger multilinear inequalities, e.g., \eqref{eq:odd} and~\eqref{eq:running}. We describe the separation algorithms for these constraints below.
 
\subsubsection{Separation Oracle for Constraints~\eqref{eq:pers:oracle-4} and \eqref{eq:pers:oracle-5}}

\begin{algorithm}[htbp]
\small 
\caption{\small Separation Oracle for Constraints~\eqref{eq:pers:oracle-4} and \eqref{eq:pers:oracle-5}   }\label{alg:x:sepa}
\begin{algorithmic}[1]
    \REQUIRE A point $(\hat{\rho}, \hat{\y},\hat{\x})$; attraction value vector $\v$; a threshold $\epsilon>0$.
    \ENSURE A set of violated constraints, $\text{Cuts}$.

    \STATE Initialize $\text{Cuts} \gets \emptyset$.
    \FOR{$i \in V$}
        \STATE Initialize $A, B,C \gets \emptyset$. 
        \FOR{each edge $e \in \{e\in \E| i\notin e\} $}  
            \IF{$\hat{y}_i + \hat{y}_e  - \hat{\rho} >0$}
                \STATE $A \gets A\cup \{e\}$ 
            \ENDIF
            \IF{$ \hat{y}_e  < \hat{y}_i $}
                \STATE $B \gets B\cup \{e\}$
            \ELSE
            \STATE $C\gets C\cup \{e\}$
            \ENDIF 
        \ENDFOR 
        \IF{$\hat{x}_i < \hat{y}_i   + \sum_{e \in \E| i\in e } \hat{y}_e v_e + \sum_{e\in A} v_e (\hat{y}_i + \hat{y}_e  - \hat{\rho}) - \epsilon$}
            \STATE $\text{Cuts} \gets \text{Cuts} \cup \{ x_i \ge y_i + \sum_{e\in \E|i\in e}y_e v_e + \sum_{e\in A}v_e (y_i +y_e -\rho)\}$.
        \ENDIF
        \IF{$\hat{x}_i > \hat{y}_i   + \sum_{e \in \E | i\in e } \hat{y}_e v_e +  \sum_{e\in B} v_e \hat{y}_e  + \sum_{e\in C} v_e \hat{y}_i  + \epsilon$}
           \STATE $\text{Cuts} \gets \text{Cuts} \cup  \{ x_i \le y_i + \sum_{e\in \E|i\in e}y_e v_e+ \sum_{e\in B}v_e y_e +   \sum_{e\in C} v_e y_i \}$.
        \ENDIF
    \ENDFOR
    \RETURN $\text{Cuts}$.
\end{algorithmic}
\end{algorithm}

 Given a point $(\hat{\rho}, \hat{\y},\hat{\x})$, the separation oracle checks whether the lower and upper bounds for each $\hat{x}_i$ constructed by $\hat{\rho}$ and $\hat{\y}$ are violated. The bounds are given by 
\begin{subequations}\label{eq:x:bound}
    \begin{alignat}{3} 
    & x_i  \ge  \hat{y}_i   + \sum_{e \in \E | i\in e} \hat{y}_e v_e + \sum_{e \in \E | i\notin e } v_e \max\{ 0, \hat{y}_i + \hat{y}_e  - \hat{\rho}  \} \label{eq:lower1} \\
    & x_i   \le \hat{y}_i + \sum_{e \in \E | i\in e } \hat{y}_e v_e + \sum_{e \in \E | i\notin e } v_e \min\{ \hat{y}_e  ,  \hat{y}_i \} \label{eq:upper1} .     
    \end{alignat}
\end{subequations}
If the current point violates the above inequalities, we add the corresponding constraint to the base model. We detail each step in Algorithm~\ref{alg:x:sepa}.

 \subsubsection{Separation Oracle for Odd-cycle Inequalities}\label{sec:sepa:odd}
The main difficulty in separating odd-cycle inequalities is that both the number of cycles $C$ and the number of odd subsets $D \subseteq C$ can grow exponentially.
\cite{barahona1986cut} reformulate the odd-cycle inequality in terms of edge variables and design a polynomial-time separation oracle based on shortest-path computations. We adapt their construction to our scaled space, where inequalities are multiplied by the no-purchase variable $\rho$.

Given a simple graph $\G(V,\E)$, we introduce a virtual node $o$ with fixed value $x_o = 1$, and connect it to every node $i \in V$. The resulting graph has nodes $V\cup \{o\}$ and edges $\{(i,o)| i \in V\} \cup \{(i,j)| (i,j)\in \E \}$. For each edge, we introduce a new variable $c_{ij}$ as follows  
\[c_{io} = x_i, c_{ij} = x_i + x_j -2z_{ij}. \]
This mapping establishes a one-to-one correspondence between $(\x,\z)$ and $\boldsymbol{c}$. Under this transformation, the odd-cycle inequality becomes
\begin{equation*}
    \sum_{(i,j)\in D} (1-c_{ij}) +  \sum_{(i,j)\in C \setminus D} c_{ij} \ge 1, C \text{ is a cycle and } D\subseteq C, |D| \text{ is odd} .
\end{equation*}
The left-hand side can be interpreted as the weight of cycle $C$, where edges in $D$ receive weight $1-c_{ij}$ and the remaining edges receive weight $c_{ij}$. The separation problem is therefore equivalent to finding a minimum-weight cycle under the constraint that an odd number of edges are selected with weight $1-c_{ij}$.

Following~\cite{barahona1986cut}, we solve this separation problem via an extended graph $\G^\prime$. For each node $i\in V\cup \{o\}$ in the original graph, $\G^{\prime}$ contains two nodes $i$ and $i^{\prime}$. $\G^{\prime}$ has four types of edges:
\[ \E^{\prime} = \bigl\{(i,j),(i^{\prime},j^{\prime}),(i,j^\prime), (i^{\prime},j)| (i,j)\in \E \cup \{(i,o)| i\in V\} \bigr\}.  \]
Given a candidate point $(\hat{\rho}, \hat{\y} )$,we define scaled edge lengths as follows:
\begin{subequations}\label{eq:weight:scale}
    \begin{alignat}{3}
         & w_{io} = w_{i^\prime o^\prime} = \hat{\rho} \cdot \hat{c}_{io} = \hat{\rho} \cdot \hat{x}_{i}= \hat{y}_i  && \forall i \in V \\ 
& w_{ij} = w_{i^\prime j^\prime} = \hat{\rho} \cdot \hat{c}_{ij} = \hat{\rho} \cdot (\hat{x}_i + \hat{x}_j -2\hat{z}_{ij}) =  \hat{y}_i  + \hat{y}_j - 2\hat{y}_{ij} && \forall (i,j) \in \E \setminus V \\
& w_{ij^{\prime}} = w_{i^{\prime}j} = \hat{\rho} \cdot (1- \hat{c}_{ij} ) =  \hat{\rho} \cdot ( 1-\hat{x_i} - \hat{x}_j + 2\hat{z}_{ij})= \hat{\rho} - \hat{y}_i  - \hat{y}_j +  2\hat{y}_{ij}  && \forall (i,j) \in \E \setminus V.
    \end{alignat}
\end{subequations} 
Edges connecting nodes of the same type correspond to weights $\hat{\rho}\hat{c}_{ij}$, while edges connecting primed and unprimed nodes correspond to weights $\hat{\rho}(1-\hat{c}_{ij})$.

We then compute, for each node $i\in V$, the shortest path from $i$ to $i^\prime$ in $\G^{\prime}$. By construction, any path from $i$ to $i'$ must traverse an odd number of ``cross" edges (i.e., edges of type $(u,v')$ or $(u',v)$). Hence, such a path corresponds exactly to selecting an odd number of edges weighted by $\hat{\rho}(1-\hat{c}_{ij})$ in the original graph. For example, a feasible path  $i\to j^{\prime}\to i^{\prime}$ 
contains one cross edge and one same-type edge, yielding a total length
\[ \hat{\rho}(1-\hat{c}_{ij}) +\hat{\rho} \hat{c}_{ij} = \hat{\rho},  \]
which matches the scaled right-hand side and therefore does not violate the inequality. Consequently, if there exists a violated odd-cycle inequality (i.e., one whose left-hand side is strictly smaller than $\hat{\rho}$), then the shortest path over all pairs $(i, i')$ will have length strictly less than $\hat{\rho}$ and corresponds to a violated cycle in the original graph. 

Computing one shortest path requires $\mathcal{O}(|V|^2)$ operations. Repeating this for all $i\in V$ yields an overall separation complexity of $\mathcal{O}(|V|^3)$. Algorithm~\ref{algo:odd} summarizes this separation procedure.

When separating odd-cycle inequalities for the Big-M formulation or the Conic formulation, we require an unscaled length $\w$ given as follows.
\begin{subequations}\label{eq:weight}
    \begin{alignat}{3}
         & w_{io} = w_{i^\prime o^\prime} =   \hat{c}_{io} = \hat{x}_i && \forall i\in V\notag \\ 
& w_{ij} = w_{i^\prime j^\prime} =   \hat{c}_{ij} = \hat{x}_i  + \hat{x}_j  - 2\hat{z}_{ij} && \forall (i,j) \in \E \setminus V \notag  \\
& w_{ij^{\prime}} = w_{i^{\prime}j} =  1- \hat{c}_{ij}  =  1 - \hat{x}_i  - \hat{x}_j +  2\hat{z}_{ij}  && \forall (i,j) \in \E \setminus V.\notag 
    \end{alignat}
\end{subequations} 
Accordingly, the length of the shortest path is compared with $1-\epsilon$ of Step~\ref{algo:odd:violate} in Algorithm~\ref{algo:odd}.

\begin{algorithm}[htbp]
\small 
\caption{\small Separation Oracle for Odd-cycle Inequality}
\label{algo:odd}
\begin{algorithmic}[1]
    \REQUIRE A point $(\hat{\rho}, \hat{\y} )$; attraction value vector $\v$; a threshold $\epsilon>0$.
    \ENSURE A set of violated constraints, $\text{Cuts}$.

    \STATE Initialize $\text{Cuts} \gets \emptyset, C \gets \emptyset$ and $D \gets \emptyset$.
    \STATE Construct the extended graph $\G^{\prime}$ and edge length $\w$ via~\eqref{eq:weight:scale}. 
    \STATE  Find shortest path $\texttt{path}(i)$ from $i$ to $i^\prime$ over all $i\in V$, denoted as $\texttt{path}^*(i)$. 
      \IF{ the length of the shortest path $\texttt{path}(i)$ is less than $\hat{\rho}(1 -\epsilon)$} \label{algo:odd:violate}
            \STATE  $C \gets C \cup (u,v) $ if  $ (u,v)$ or $(u^{\prime}, v^{\prime})  \in \texttt{path}(i)$ 
            \STATE $D \gets D \cup   (u,v) $ if  $  (u,v^{\prime})$ or $ (u^{\prime},v)  \in \texttt{path}(i)$ 
          \STATE  $\text{Cuts} \gets  \text{Cuts}  \cup \eqref{eq:odd}$ with cycle $C$ and odd-number subset $D$.
        \ENDIF   
    \RETURN $\text{Cuts}$.
\end{algorithmic}
\end{algorithm}

 \subsubsection{Separation Oracle for Running-Intersection Inequalities}\label{app:sep-RI} Finding a violated running-intersection inequality is nontrivial because a multi-purchase hypergraph can admit an exponential number of such inequalities. \citet{del2020impact} propose a separation oracle with complexity $O(|\E|(d2^d|\E|+2^{d\bar{m}}\bar{m}^2d))$, where $\bar{m}$ denotes the maximum number of neighbors of a central bundle $e_0$ considered by the algorithm. In many practical applications, the rank $d$ is relatively small, i.e., $d\le 3$ in our case study. Hence, this algorithm is polynomial in the number of candidate bundles $|\E|$, making it computationally efficient for practical instances. We illustrate this algorithm for our perspective formulation as follows. 

 We begin by constructing candidate running-intersection structures. Given a multi-purchase hypergraph $\G$, we identify a collection
 \[S=\{(e_0, O_{e_0}, P_{e_0})\},\]
 where $e_0$ is a central bundle. The set $O_{e_0}$ consists of running-intersection orderings $\tilde{e}_1,\dots, \tilde{e}_k$ intersected with $e_0$. A candidate running-intersection ordering $\tilde{e}_1,\dots, \tilde{e}_m$ satisfies: for any $f,g\in [m]$, $\tilde{e}_f \not \subset \tilde{e}_g $ and $\tilde{e}_f \not \supset \tilde{e}_g $.
This condition is imposed to control the number of the resulting inequalities while preserving tightness. Intuitively, when $\tilde{e}_f  \subset \tilde{e}_g $, the corresponding running-intersection inequality is redundant: it can be obtained by aggregating a running-intersection inequality involving only $\tilde{e}_f$ together with a trivial constraint $  z_{e_g} \le z_{\nu_g}$, where $\nu \in N(\tilde{e}_g)$. Excluding such nested intersections, therefore, leads to more compact inequalities without loss of strength. The mapping $P_{e_0}$ assigns to each intersection $\tilde{e}_k$ a set of neighbors, so that for any $e_k \in P_{e_0}(\tilde{e}_k)$, we have $e_k\cap e_0 = \tilde{e}_k$. In some cases, such as Figure~\ref{fig:ric}, $P_{e_0}(\tilde{e}_k)$ may include only one element.  Algorithm~\ref{algo:ri_structure} details the construction of all admissible running-intersection structures.

 \begin{algorithm}[htbp]
 \small 
\caption{\small Construction of the Running Intersection Separation Structure}
\label{algo:ri_structure}
\begin{algorithmic}[1]
    \REQUIRE A multi-purchase hypergraph $\G$; a maximum number of neighbors $\bar{m}$.
    \ENSURE Running-intersection structures.

    \STATE Initialize structure $S \gets \emptyset$.
    \FOR{$e_0 \in \E$}
        \STATE Initialize $O_{e_0} \gets \emptyset$, $P_{e_0} \gets \emptyset$, and the subgraph edge set $\tilde{\E}_{e_0} \gets \emptyset$. 
        \STATE For each $e_k \in \E \setminus \{e_0\}$, let $f \leftarrow e_0 \cap e_k$, then $\tilde{\E}_{e_0} \gets \tilde{\E}_{e_0} \cup  f$ and add $e_k$ to $P_{e_0}(f)$.  
        \FOR{each subset $ H \subseteq \tilde{\E}_{e_0} $ with $|H|\le \bar{m}$}   
            \IF{ for any $\tilde{e}_f, \tilde{e}_g\in H$, $\tilde{e}_f \not \subset \tilde{e}_g $ and $\tilde{e}_f \not \supset \tilde{e}_g $.}
                \STATE If $H$ admits a running-intersection ordering, find such an ordering and add it to $O_{e_0}$.
            \ENDIF 
        \ENDFOR
        \STATE $S\gets S \cup (e_0, O_{e_0} , P_{e_0})$
    \ENDFOR
    \RETURN $S$.
\end{algorithmic}
\end{algorithm}

 Given running-intersection structures and a candidate solution $(\hat{\rho}, \hat{\y})$, the separation algorithm tries to find the most violated one. Recall that the scaled running-intersection inequality is 
 \[\sum_{k\in [m]} y_{e_k} + \sum_{v \in e_0\setminus \bigcup_{k\in [m]}e_k} y_{v} \le  y_{e_0}+ \sum_{k\in [m]| N(e_0\cap e_k) \neq \emptyset} y_{\mu_k} + \rho (\omega - 1). \]
For each central bundle $e_0$ and an ordering $\tilde{e_1},\dots, \tilde{e}_m$, we can calculate $\omega$ by its definition. Hence, the separation task reduces to maximizing the left-hand side and minimizing the right-hand side variables. Specifically, for $k\in [m]$, we find the associated neighbor with the largest value $\bar{e}_{k} = \argmax_{ c \in P_{e_0}(\tilde{e}_k)}~\hat{y}_{c}$ 
and the smallest single element in the intersection $\bar{\mu}_{k} = \argmin_{c \in N(\tilde{e}_k)} ~\hat{y}_{c}$.
If the corresponding LHS is larger than the RHS, then this running-intersection inequality should be added to the base model. We detail the separation process in Algorithm~\ref{algo:ri_sep}.

\begin{algorithm}[ht]
\small 
\caption{\small Separation Oracle for Running Intersection Inequality}
\label{algo:ri_sep}
\begin{algorithmic}[1]
     \REQUIRE A point $(\hat{\rho}, \hat{\y} )$; attraction value vector $\v$; a threshold $\epsilon>0$.
    \ENSURE A set of violated constraints, $\text{Cuts}$.
    \FOR{$e_0 \in \E$}
    \FOR{$(\tilde{e}_1,\dots, \tilde{e}_m) \in O_{e_0}$}
        \STATE $\omega = |e_0\setminus \bigcup_{k\in [m]}\tilde{e}_{k}|+ |\{k\in [m] :  N(\tilde{e}_k) =\emptyset \}|$
        \STATE $\bar{e}_{k} = \argmax_{ c \in P_{e_0}(\tilde{e}_k)}~\hat{y}_c$ for each $k\in [m]$
        \STATE $\bar{\mu}_{k} = \argmin_{c \in N(\tilde{e}_k)} ~\hat{y}_{c}$ for each $k\in [m]$ if $N(\tilde{e}_k)\neq \emptyset$
        \IF{ $ \sum_{k\in [m]} \hat{y}_{\bar{e}_k} + \sum_{v\in e_0\smallsetminus \bigcup_{k\in [m]} \bar{e}_k} \hat{y}_{v}> \hat{y}_{e_0} + \sum_{k\in [m]| N(e_0\cap \bar{e}_k) \neq \emptyset} \hat{y}_{\bar{\mu}_k} + \hat{\rho} (\omega - 1) + \epsilon $}
        \STATE  $\text{Cuts} \gets  \text{Cuts} \cup \{\sum_{k\in [m]} {y}_{\bar{e}_k} + \sum_{v\in e_0\smallsetminus \bigcup_{k\in [m]}\bar{e}_k } y_v \le y_{e_0} + \sum_{k\in [m]| N(e_0\cap \bar{e}_k) \neq \emptyset} {y}_{\bar{\mu}_k} + {\rho} (\omega - 1)\} $ 
        \ENDIF
     \ENDFOR
    \ENDFOR
    \RETURN $\text{Cuts}$.
\end{algorithmic}
\end{algorithm}

 \subsection{Cutting-Plane Implementation}\label{sec:cut} 
We implement a cutting-plane procedure to solve optimization problems over~\eqref{eq:pers:oracle}. We initialize the procedure with the exact base formulation~\eqref{eq:base}. The algorithm iteratively solves the LP relaxation of the current model and augments it with valid inequalities identified by the separation oracles proposed in Section~\ref{sec:base}. We set the maximum iterations limit to $K = 500$ and the tolerance parameter to $\epsilon = 1 \times 10^{-5}$; a cut is added only if its violation exceeds this threshold. The process terminates when no new cuts are detected, the iteration limit is reached, or the cumulative separation time exceeds $T = 800$ seconds for the perspective formulation and $T=3600$ seconds for benchmark formulations. Let $\texttt{Base}_k$ denote the formulation obtained at the $k^{\text{th}}$ iteration. We solve the final formulation with binary constraints to recover an optimal integer solution. Algorithm~\ref{alg:framework} details this procedure.
 
\begin{algorithm}[htbp] 
\small
\caption{\small Cutting-Plane Implementation}
\label{alg:framework}
\begin{algorithmic}[1]
    \REQUIRE A base formulation $\texttt{Base}_0$ \eqref{eq:base}; maximum iterations $K$; tolerance $\epsilon$; time limit $T$.
    \ENSURE An optimal integer solution.
 
    \STATE Record $t_{\texttt{start}} \leftarrow \text{CurrentTime}$.
    
    \FOR{$k \leftarrow 1$ to $K$}
        \STATE $(\hat{\rho}, \hat{\y}, \hat{\x}) \leftarrow \text{SolveLP}(\texttt{Base}_{k-1})$.
        \STATE Cuts $\leftarrow$ Call separation oracles Algorithm~\ref{alg:x:sepa}, Algorithm~\ref{algo:odd}, and Algorithm~\ref{algo:ri_sep} for solution $(\hat{\rho}, \hat{\y}, \hat{\x})$.
        \STATE $\texttt{Base}_{k} \leftarrow \text{Add Cuts to } \texttt{Base}_{k-1}$
        \STATE If Cuts = $\emptyset$ or $\text{CurrentTime} - t_{\texttt{start}} > T$ then break
    \ENDFOR
    \STATE Impose binary constraints on variables in $\texttt{Base}_{k}$ and solve the MIP formulation.  
    \RETURN $(\rho^*, \y^*, \x^*)$.
\end{algorithmic}
\end{algorithm}

\end{document}